\documentclass[onefignum]{siamart220329}

\usepackage{amsmath}
\usepackage{amssymb}
\usepackage[english]{babel}
\usepackage{color}
\usepackage{enumitem}   
\usepackage{graphicx}
\usepackage[utf8]{inputenc}
\usepackage{subfigure}
\usepackage{tikz,pgfplots}
\usepackage{wrapfig}
\usepackage{xargs}
\usepackage{xcolor}

\newcommand{\dint}[1]{\;\mathrm{d} #1 }
\newcommand{\R}{\mathbb{R}}
\newcommand{\N}{\mathbb{N}}

\newcommand{\Lp}[2]{\mathrm{L}^{ #1 }\left( #2 \right)} 
\newcommand{\Lpw}[3]{\mathrm{L}_{#2}^{ #1 }\left( #3 \right)}

\newcommand{\C}[2]{\mathcal{C}^{#1}\left( #2 \right) }

\newcommand{\inner}[3][]{ \left< #2, #3\right>_{ #1 } }
\newcommand{\einner}[1]{ \left< \cdot, \cdot \right>_{ #1 } }

\theoremstyle{plain}
\newtheorem{remark}[theorem]{Remark}
\newtheorem{assumption}[theorem]{Assumption}

\DeclareMathOperator*{\divergence}{div}

\DeclareMathOperator*{\esssup}{ess\,sup}
\DeclareMathOperator*{\essinf}{ess\,inf}
\DeclareMathOperator*{\tr}{tr}
\DeclareMathOperator*{\Real}{Re}

\DeclareMathOperator*{\range}{range}
\mathchardef\ordinarycolon\mathcode`\:
\mathcode`\:=\string"8000
\begingroup \catcode`\:=\active
\gdef:{\mathrel{\mathop\ordinarycolon}}
\endgroup
 
\ifpdf 
\hypersetup{
  pdftitle={On the approximability of Koopman-based operator Lyapunov equations},
  pdfauthor={T. Breiten and B. H\"oveler}
}
\fi

\title{On the approximability of Koopman-based operator Lyapunov equations\thanks{Submitted to the editors February 10, 2023.
}}

\author{Tobias Breiten\thanks{Institute of Mathematics, Technische Universit\"at Berlin, Stra\ss e des 17. Juni 136, 10623 Berlin, Germany
  (\email{tobias.breiten@tu-berlin.de, bernhard.hoeveler@tu-berlin.de}).}
\and Bernhard Höveler\footnotemark[2] 
}

\begin{document}

	\maketitle
	\begin{abstract}
	Lyapunov functions play a vital role in the context of control theory for nonlinear dynamical systems. Besides its classical use for stability analysis, Lyapunov functions also arise in iterative schemes for computing optimal feedback laws such as the well-known policy iteration. In this manuscript, the focus is on the Lyapunov function of a nonlinear autonomous finite-dimensional dynamical system which will be rewritten as an infinite-dimensional linear system using the Koopman or composition operator. Since this infinite-dimensional system has the structure of a weak-* continuous semigroup, in a specially weighted $\mathrm{L}^p$-space one can establish a connection between the solution of an operator Lyapunov equation and the desired Lyapunov function. It will be shown that the solution to this operator equation attains a rapid eigenvalue decay which justifies finite rank approximations with numerical methods. The potential benefit for numerical computations will be demonstrated with two short examples.
	\end{abstract}
	
	% REQUIRED

	% REQUIRED
	
	\section{Introduction}
	\label{sec:intro}
	
	We consider nonlinear dynamical systems of the form
	\begin{equation}\label{eq:dyn_sys}	
	\left\{ \begin{array}{rclcc} 
		\frac{\mathrm{d}}{\mathrm{d}t} x(t) &=& f(x(t)), & \qquad & \text{for} \; t \in (0,\infty),\\
		x(0) &=& z, & &
	\end{array} \right.
	\end{equation}
	where $z \in \mathbb R^{d}$ and $f\in \C{1}{\R^{d}; \R^{d}}$. If for given $z\in \mathbb R^{d}$ there exists a unique solution $x(\cdot)$ to \eqref{eq:dyn_sys}, we use the notation $x(t):=\Phi^t(z)$. Our interest is the computation of the cost functional	  
	\begin{align}\label{eq:cost_func}   
	  v(z) := \int_0^\infty g(\Phi^t(z)) \dint t
	\end{align}
	for some given $g\colon \mathbb R^{d}\to \mathbb R$. In what follows, we will restrict ourselves to initial values $z \in \overline{\Omega} \subset \mathbb R^{d}$, where $\Omega$ is a bounded open domain with $\mathcal{C}^1$ boundary. In particular, we assume that $\Omega$ is \emph{flow-invariant} under the system \eqref{eq:dyn_sys}, i.e., for every $z \in \overline{\Omega}$ it holds that $\Phi^t(z)\in \overline{\Omega}$ for all $t\ge 0$. It is well-known, see, e.g., \cite[Chapter III, Theorem XVI]{Wal98} that this is guaranteed if $f$ is continuous on $\overline{\Omega}$ and satisfies a tangent condition of the form
	\begin{align}\label{eq:tan_con}   
	 \nu(x)^\top f(x) \le 0 \ \ \text{for all } x\in \partial \Omega,
	\end{align}
	where $\nu(x)$ denotes the outer unit normal to the boundary $\partial \Omega$. Let us moreover emphasize that this implies that the solution $\Phi^t(z)$ exists for all $t\ge 0$ so that the cost functional \eqref{eq:cost_func} is well-defined as a mapping from $\overline{\Omega}$ to $[0,\infty]$. If $g(z)=\|h(z)\|^2$ and \eqref{eq:dyn_sys} is locally asymptotically stable around the origin, $v$ is characterized by the first order nonlinear partial differential equation (PDE)
	\begin{align}\label{eq:lyap_pde}   
	 \nabla v(z)^\top f(z) + \|h(z)\|^2=0, \ \ v(0)=0,
	\end{align}
	where $\nabla v(z)=(\tfrac{\partial v}{\partial z_1},\dots, \tfrac{\partial v}{\partial z_{d}})^\top$, see, e.g., \cite[Theorem 3.2]{Sch93}. If, additionally, the system is linear and the costs are quadratic, i.e., $f(z)=Az$, $g(z)=z^\top C^\top Cz$ then $v(z)=z^\top Pz$ with $P$ being the unique symmetric positive semidefinite solution to the observability Lyapunov equation
	\begin{align}\label{eq:lyap_fin_dim}   
	 A^\top P + P A + C^\top C=0.
	\end{align}
	In fact, similar results also hold true for the case of infinite-dimensional linear systems, see \cite[Theorem 4.1.23]{CurZ95} and one of the main ideas of this article is to use the known Koopman embedding which replaces \eqref{eq:dyn_sys} by an infinite-dimensional linear system such that \eqref{eq:lyap_pde} can be related to an operator Lyapunov equation similar to \eqref{eq:lyap_fin_dim}.\\
	
	\paragraph{Existing literature and related results} 
	 
	Computing Lyapunov functions for linear systems has been studied extensively in the literature, see the detailed overviews \cite{BenS13,Sim16} and the references therein. In particular, the so-called large-scale case where the system dynamics $f(x)=Ax$ are associated with a high dimensional system resulting from a spatial semi discretization of a PDE has received much attention. The efficacy of numerical methods here relies on the nowadays well-known fact that the solution $P$ to \eqref{eq:lyap_fin_dim} often exhibits a very fast singular value decay which can be exploited with low rank techniques \cite{BenLP08,OpmRW13,Sim07}. For some early works that discuss such properties from a finite-dimensional perspective, we refer to \cite{AntSZ02,Gra04,Pen00}. Beginning with \cite{CurS01}, considerable progress such as nuclearity of the solution operator $P$ or $p$-summability of the singular values has also been made from an infinite-dimensional perspective, see \cite{GruK14,Opm10,Opm15}. Most of the previous results rely on (spectral) properties of the generator $A$ of the underlying system and therefore restrict to a particular class of systems such as analytic control systems. Very recently, in \cite{Opm20} the author has obtained an approximation result for solutions to operator Lyapunov equation which is not based on analytic semigroup theory and therefore covers the case of hyperbolic PDEs. One of the main ideas in that article is to compensate for the lack of regularity of the solution by means of a particularly regularizing observation operator which, in the context of \eqref{eq:cost_func}, can be interpreted as a specific cost function $g$. We will follow a similar strategy which we elaborate upon later in this article. The general idea of embedding nonlinear dynamics in an infinite-dimensional system has a longstanding tradition with its origin tracing back to at least \cite{Koo31} and \cite{Car32}. A renewed interest, specifically with regard to applications in control theory, goes back to \cite{Mez05} and has inspired a great amount of work in the recent literature. While a full overview on aspects of Koopman and composition operators is beyond the scope of this article, we refer to the overview articles \cite{Bruetal22,BudMM12,Kluetal20} and the monograph \cite{MauSM20} as well as the references therein. Let us further mention \cite{MauM16} where the authors discuss nonlinear stability analysis by inspection of the eigenfunctions of the infinitesimal generator of the Koopman semigroup. 
	
	\paragraph{Contribution} Following the aforementioned embedding of \eqref{eq:dyn_sys} into an infinite-dimensional linear system, in this article we will discuss a functional analytic setting which allows to express \eqref{eq:cost_func} implicitly via a solution of an abstract operator Lyapunov equation on an appropriately weighted $\mathrm{L}^p$ space. Our main results can be summarized as follows:
	\begin{itemize}   
	 \item[(i)] Since the Koopman semigroup is not strictly contractive, we utilize a weight function as in \cref{assumption:comp_weight} to show that the composition semigroup becomes exponentially stable on the associated weighted $\mathrm{L}^p$ space, see \cref{thm:exp_stab}.
	 \item[(ii)] Following the linear quadratic case, we define a candidate $P$ to replace \eqref{eq:cost_func} by the abstract bilinear form \eqref{defi:value_bilinear_form} which we show in \cref{thm:spectral_decay} to be approximable by  convergent finite rank operators. Here, the approximation rate will depend on the structure and smoothness of the cost function $g$.
	 \item[(iii)] For the sum of squares solution (see \cref{defi:sum_of_squares}) induced by the eigenfunctions of the operator $P$, we show in \cref{theorem:sum_of_squares_is_normal} that it coincides with \eqref{eq:cost_func} by means of a Dirac sequence.
	 \item[(iv)] The operator $P$ is shown to satisfy an operator Lyapunov equation in \cref{thm:operator_lyap}. 
	\end{itemize}
 
	The precise structure of this article is as follows. After a brief review of well-known results on Koopman or composition operators and weighted $\mathrm{L}^p$ spaces, in \cref{sec:comp_sg} we replace the nonlinear dynamics \eqref{eq:dyn_sys} by an infinite-dimensional exponentially stable weak-* continuous semigroup with infinitesimal generator $A$. \Cref{sec:cost_operator_and_sum_of_squares} introduces a specific structure of the cost function $g$ which allows for an interpretation of the extended observability map arising in the context of well-posed linear systems. \Cref{sec:op_lyap} contains the characterization of $P$ as the solution to an operator Lyapunov equation. In \cref{sec:numerics}, we illustrate our numerical findings by means of two numerical examples. A short conclusion with an outlook for future research is provided in \cref{sec:conclusion}.\\
	
	\textbf{Notation.}
	For a Banach space $X$, we denote its topological dual space by $X^\ast$ and by $\einner{X,X^\ast}$ the dual pairing between $X$ and $X^\ast$. In the case of a Hilbert space $H$ we simply write $\einner{H}$ for the dual pairing. The space of bounded linear operators mapping from $X$ to itself is denoted by $\mathcal{L}(X)$. For a linear (unbounded) operator $A$ with domain $\mathcal{D}(A)$ in $X$ mapping to $Y$ we write $A \colon \mathcal{D}(A) \subseteq X \to Y$. If $\mathcal{D}(A)$ is dense in $X$, the adjoint  of such an operator is denoted by $A^\ast\colon \mathcal{D}(A^*)\subseteq Y^*\to X^*$. With $\sigma( A ) \subseteq \mathbb{C}$ we denote the spectrum of an operator. For a set $\Omega \subseteq \R^{d}$ we denote the closure by $\bar{\Omega}$ and the interior by $\Omega^\circ$. By $\C{m}{\Omega}$ we denote the set of $m$-times continuously differentiable functions over $\Omega$. The Lebesgue space to an index $1 \le p \le \infty$ and a Banach space $Y$ is denoted by $\Lp{p}{\Omega; Y}$. If $Y = \R$ we write $\Lp{p}{\Omega}$. The Sobolev space  to an index $k \in \N$ and $1 \le p \le \infty$ over a set $\Omega \subseteq \R^{d}$ is denoted by $W^{k,p}(\Omega)$. The Jacobi matrix containing all first order derivaties is denoted by $\mathrm{D}$. For a matrix $A \in \R^{d\times d}$ we denote the eigenvalues by $\lambda_i(A)$. If the matrix $A$ is symmetric, i.e., $A=A^\top$, we write $\lambda_{\text{min}}(A)$ and $\lambda_{\text{max}}(A)$ to denote the smallest and largest eigenvalue, respectively.
	
	\section{The composition semigroup and its adjoint}
	\label{sec:comp_sg}
	
	In this section, we first recall some well-known facts about composition operators on (weighted) Lebesgue spaces. In particular, we review existing results on the Koopman operator and its left-adjoint, the transfer or Perron-Frobenius operator. With the intention of relating the Lyapunov function \eqref{eq:cost_func} to a specific operator Lyapunov equation, we introduce an appropriately weighted $\Lpw{p}{w}{\Omega}$-space on which the composition semigroup associated with the flow $\Phi^t(z)$ of \eqref{eq:dyn_sys} will turn out to be exponentially stable.
		
	\subsection{Koopman and Perron-Frobenius operators}
	
	Instead of the nonlinear finite-dimensional system \eqref{eq:dyn_sys} which describes pointwise dynamics, one might focus on an infinite-dimensional linear formulation induced by a so-called \emph{composition operator}. The following results are well-known in the literature and can be found in many textbooks such as, e.g., \cite{LasM94}. 

	The Koopman operator can be seen as acting on observables evaluated along the flow $\Phi^t(z)$ of the system \eqref{eq:dyn_sys}. For a function space $Y$ and observables $Y^\ast \ni \psi\colon \Omega \to \mathbb R$, the Koopman operator $\mathcal{K}^t\colon Y^\ast \to Y^\ast$ associated with \eqref{eq:dyn_sys} is defined for fixed $t$ by
	\begin{align*}   
	 \psi\mapsto \mathcal{K}^t\psi = \psi\, \circ\, \Phi^t, \ \ (\mathcal{K}^t\psi)(z) =\psi(\Phi^t(z)).
	\end{align*}
	Depending on the nature of the dynamical system and the chosen function space $Y$, the family $\{\mathcal{K}^t\}_{t\ge 0}$ of Koopman operators enjoys additional properties. In fact, since the composition of functions is linear and \eqref{eq:dyn_sys} is time-invariant, for $Y=\Lp{1}{\Omega}$ and $Y^\ast = \Lp{\infty}{\Omega}$ respectively, the Koopman operator $\mathcal{K}^t$ defines a semigroup of bounded linear operators on $Y^\ast$, i.e., we have 
	\begin{itemize}   
	 \item[(i)] $\mathcal{K}^t\in\mathcal{L}(Y^\ast),\ t\ge 0,$
	 \item[(ii)] $\mathcal{K}^0=\mathrm{id}_{Y^\ast},$
	 \item[(iii)] $\mathcal{K}^{t+s}=\mathcal{K}^t\mathcal{K}^s,t,s\ge 0.$
	\end{itemize}
	Note however that on $\Lp{\infty}{\Omega}$, the semigroup is generally not strongly continuous \cite[Theorem 7.4.2]{LasM94} but only weak-* continuous, see also the discussion in \cref{subsec:comp_lpw}.
	
	If one is interested in the statistical behavior, it is useful to study how probability densities $\rho$ evolve under the dynamics \eqref{eq:dyn_sys}. This naturally leads to the transfer or \emph{Perron-Frobenius} operator $\mathcal{P}^t$ which is defined by
		\begin{align*}   
		 \rho\mapsto \mathcal{P}^t\rho,  \ \ (\mathcal{P}^t\rho)(z) =\rho(\Phi^{-t}z)| \det{\mathrm{D}\Phi^{-t}z}| .
		\end{align*}
	
	In this case, a canonical choice for the function space $Y$ is $\Lp{1}{\Omega}$ on which $\mathcal{P}^t$ becomes a strongly continuous (stochastic) semigroup, see, e.g., \cite[Section 7.4]{LasM94}, meaning that $\mathcal{P}^t$ is a positivity preserving contraction semigroup. Let us emphasize that $\mathcal{P}^t$ is  an isometry with its eigenvalues being located on the unit circle, see \cite[Corollary 2.5]{FroJK18}.
	
	The operators $\mathcal{K}^t$ and $\mathcal{P}^t$ are adjoint to each other, i.e., for all $\rho\in  \Lp{1}{\Omega},\psi \in \Lp{\infty}{\Omega}$ and $t\ge0$, we have
	\begin{align*}   
	 \langle \mathcal{P}^t \rho,\psi \rangle_{\Lp{1}{\Omega},\Lp{\infty}{\Omega}} =  \langle  \rho,\mathcal{K}^t\psi \rangle_{\Lp{1}{\Omega},\Lp{\infty}{\Omega}}.
	\end{align*}
	Turning to the infinitesimal generators $\mathcal{A}_{\mathcal{K}}$ and $\mathcal{A}_{\mathcal{P}}$, for $f$ as in \eqref{eq:dyn_sys}, we have the characterization (\cite[Section 7.6]{LasM94})
	\begin{align*}   
	 \mathcal{A}_{\mathcal{K}} \psi = f^\top \nabla \psi, \quad \mathcal{A}_{\mathcal{P}} \rho = - \mathrm{div}(\rho f).
	\end{align*}
	In other words, $\mathcal{A}_{\mathcal{K}}$ and $\mathcal{A}_{\mathcal{P}}$ are hyperbolic first-order differential operators. Note that it is also common to assume that \eqref{eq:dyn_sys} is only accurate up to small stochastic perturbations in form of a white noise term which renders the resulting generators parabolic, see, e.g., \cite{FroJK18}. In this case, both Koopman and Perron-Frobenius operators are frequently considered on weighted spaces which are then assumed to be related to the invariant probability density of the stochastic dynamics. Here, we will restrict to the fully deterministic and thus hyperbolic case.
	 
	\subsection{Weighted $\Lp{p}{\Omega}$ spaces}
	
	Here, we briefly recall the theory of weighted Lebesgue spaces. The material is rather standard and can be found in any standard textbook, e.g., \cite{Alt12}. 
	For a given measurable \emph{weight function} 
	\begin{align}\label{eq:weighting}    
	w\colon \Omega \to \R_+\quad \text{with } w^{-1} \in W^{1,\infty}(\Omega)
	\end{align}
	we define the following weighted Lebesgue spaces.
	\begin{definition}[The space $\Lpw{p}{w}{\Omega}$]
		Given a weight function $w$ as in \eqref{eq:weighting}, we define
		\[
		\| \phi \|_{\Lpw{p}{w}{\Omega}} := \left\{ \begin{array}{ccl}\left(
			\int_{\Omega} |\phi(x)|^p  w(x) \dint x \right)^{\frac{1}{p}}, & \qquad & \text{for} \; 1 \le p < \infty, 
			\vspace{.2cm}\\
			\underset{ x \in \Omega} { \esssup } \left| \phi(x) w(x)  \right|, & \qquad & \text{for} \; p = \infty 
		\end{array}\right.
		\]	
		and the corresponding space $\Lpw{p}{w}{\Omega}$ as
		\begin{equation*}
			\Lpw{p}{w}{\Omega} := \left\{ \phi \colon \Omega \to \R \; \big| \; \| \phi \|_{\Lpw{p}{w}{\Omega}} < \infty \right\}.
		\end{equation*}
	\end{definition}
	Note that in the case $p = \infty$ our definition does not coincide with the usual definition for Lebesgue spaces from, e.g., \cite[Definition 3.15]{Alt12} where the weighting $w$ is not included within the essential supremum. Our modification is motivated by a different dual pairing which we utilize frequently throughout the rest of this article. We now have the following result.
	\begin{lemma}
		Let $X = \Lpw{p}{w}{\Omega}$ for some $1 \le p < \infty$ and $\varphi \in X^\ast$ then it holds that
		\[
		\inner[X,X^\ast]{\phi}{\varphi} = \left\{ \begin{array}{lcl}
			\int_{\Omega} \phi(x) {\varphi}(x) w(x) \dint x & \qquad & \text{for} \; 1 < p < \infty ,
			\vspace{.2cm}
			\\
			\int_{\Omega} \phi(x) \varphi(x) w^2(x) \dint x & \qquad & \text{for} \; p = 1,			
		\end{array}		\right.
		\]
		where $ \varphi \in \Lpw{p^\ast}{w}{\Omega}$ and $ \frac{1}{p} + \frac{1}{p^\ast} = 1$.
	\end{lemma}
	\begin{proof}
		 For $1<p<\infty$, the statement is well-known \cite[Theorem 6.12]{Alt12}. For $p=1$, the assertion also follows from \cite[Theorem 6.12]{Alt12} by straightforward modification. 
	\end{proof}
	From now on we will identify elements $\varphi \in \left( \Lpw{1}{w}{\Omega} \right)^\ast$ with $\varphi \in \Lpw{\infty}{w}{\Omega}$, i.e., we identify $\left( \Lpw{1}{w}{\Omega} \right)^\ast \equiv \Lpw{\infty}{w}{\Omega}$. As $\Lpw{1}{w}{\Omega}$ and $\Lpw{\infty}{w}{\Omega}$ lack a Hilbert space structure the following embeddings will become useful.
	\begin{lemma}\label{lem:lpw_embeddings}
		For $w$ as in \eqref{eq:weighting}, it holds that
		\[
		\Lpw{\infty}{w}{\Omega} \subset \Lpw{2}{w^2}{\Omega} \subset \Lpw{1}{w}{\Omega}
		\]
		where the embeddings are dense.
	\end{lemma}
	\begin{proof}
		With the Hölder inequality we get 
		\begin{align*}
			\| \phi \|_{\Lpw{1}{w}{\Omega}} = \int_{\Omega} | \phi(x) | w(x) \dint x
			\le  \| \phi(x) w(x) \|_{\Lp{2}{\Omega}} | \Omega |^{1/2} 	= C(\Omega) \| \phi \|_{\Lpw{2}{w^2}{\Omega} }
		\end{align*}
		as well as 
		\begin{align*}
			\| \phi \|^2_{\Lpw{2}{w^2}{\Omega}} =& \int_{\Omega} | \phi(x) |^2 w(x)^2 \dint x 
			\le  \| \phi(x) w(x) \|^2_{\Lp{\infty}{\Omega}} | \Omega | = C(\Omega) \| \phi \|^2_{\Lpw{\infty}{w}{\Omega} }
		\end{align*}
		this proves $\Lpw{\infty}{w}{\Omega} \subseteq \Lpw{2}{w^2}{\Omega} \subseteq \Lpw{1}{w}{\Omega}$. To show that $\Lpw{\infty}{w}{\Omega}$ is dense in $\Lpw{1}{w}{\Omega}$, let $\phi \in \Lpw{1}{w}{\Omega}$. We define
		$
			\phi_m := \chi_{ \left\{ \phi w \le m \right\} } \phi 
		$
		for some $m \in \N$ and it follows 
		\begin{align*}
			\| \phi - \phi_m \| =  \| \chi_{ \{ \phi w > m\} } \phi  \|_{\Lpw{1}{w}{\Omega}} = \int_\Omega \chi_{ \{ \phi w > m\} }(x) |\phi(x)| w(x) \mathrm{d} x.
		\end{align*}
		For almost every $x \in \Omega$ we have $( \chi_{ \{ \phi w > m\} } \phi w )(x) \rightarrow 0$ and $|\chi_{\{ \phi > m\} } \phi w |(x) \le |\phi w|(x)$. By the dominated convergence theorem \cite[A.3.21]{Alt12} it follows 
		\[
			\lim_{m \rightarrow \infty} \| \phi - \phi_m \|_{\Lpw{1}{w}{\Omega}} = 0
		\]
		The same construction can be used to show that $\Lpw{\infty}{w}{\Omega}$ is dense in $\Lpw{2}{w^2}{\Omega}$.
		Lastly, by the inequality shown above $\phi_m \in \Lpw{2}{w^2}{\Omega}$ and therefore $\Lpw{2}{w^2}{\Omega}$ is also dense in $\Lpw{1}{w}{\Omega}$. 
	\end{proof}
In one of the later results, namely \cref{theorem:lyapunov:semigroup_property}, we will need some density result that is given in the following \cref{lemma:w1inf_dense}. Its proof consists mainly of standard textbook arguments which can be found for example in \cite[Chapter 4.4]{Eva98} and which are adapted to the weighted space.
\begin{lemma} If $\frac{1}{w} \in W^{m,\infty}(\Omega)$ for some $m \in \N$ then $W^{m,\infty}(\Omega)$ is dense in $\Lpw{\infty}{w}{\Omega}$ and $\Lpw{2}{w^2}{\Omega}$ with respect to the weak-* topology.
	\label{lemma:w1inf_dense}
\end{lemma}
\begin{proof}
	Let us define the extension 
	\[
	E \colon \Lp{p}{\Omega} \to \Lp{p}{\R^{d}}, \quad \text{with} \quad E \phi := \left\{ \begin{array}{lcc}
		\phi(x) & \qquad & \text{for $x \in \Omega$}\\
		0 & \qquad & \text{else}
	\end{array}\right. .
	\]
	Let us start with $X^*=\Lpw{\infty}{w}{\Omega}$.
	For $\psi \in \Lpw{\infty}{w}{\Omega}$ we define
	\[
	\psi_{k} := \underbrace{( \eta_{1/k} \ast E ( \psi w ) )\big|_{\Omega}}_{ \in \C{\infty}{\bar{\Omega}}} \frac{1}{w} \in W^{m,\infty}(\Omega)
	\]
	where $\eta_{\varepsilon}$ denotes the standard mollifier from \cite[Chapter 4.4]{Eva98}. Now let $\phi \in X=\Lpw{1}{w}{\Omega}$ then
	\begin{align*}
		| \inner[X,X^\ast]{\phi}{\psi_k \hspace{-1px} - \hspace{-1px} \psi}\hspace{-2px}  | = &  \big| \hspace{-2px} \int_\Omega  \phi(x) \left( \psi_k (x)  \hspace{-1px} - \hspace{-1px} \psi(x) \right)  w(x)^2 \mathrm{d}x \big|\\
		= &  \big|  \hspace{-2px} \int_{\Omega} \phi(x)\! \int_{\R^n} \hspace{-5px}  E (\psi w)(y) \eta(y \hspace{-1px} - \hspace{-1px}x) \dint y \; w(x) \dint x - \hspace{-3px} \int_{\Omega} \phi(x) \psi(x) w(x)^2  \dint x  \big|. \\
		\intertext{We can utilize Fubini \cite[A6.10]{Alt12} to show that} 
		| \inner[X,X^\ast]{\phi}{\psi_k \hspace{-1px} - \hspace{-1px}\psi}\hspace{-2px} | = & \big| \int_{\R^n}\hspace{-5px} E (\psi w) (y) \int_{\Omega}  \phi(x) \eta(y \hspace{-1px} - \hspace{-1px}x) w(x) \dint x \dint y - \hspace{-3px} \int_{\Omega} \phi(x) \psi(x) w(x)^2  \dint x  \big| \\
		= & \big| \int_{\Omega} \left(  \int_{\R^n}  E (\phi w )(x) \eta(y-x) \dint x - \phi(y) w(y) \right) \psi(y) w(y)  \dint y \; \big| \\
		\le & \| \eta_{1/k} \ast E ( \phi w ) - E ( \phi w ) \|_{\Lp{1}{\R^n}} \| \psi w \|_{\Lp{\infty}{\Omega}} .
	\end{align*}
	We know that $ \phi w  \in \Lp{1}{\Omega}$ and that there exists a compact $V$ such that $U \subset V \subset \R^{d}$. Thus, with \cite[Appendix, Theorem 7]{Eva98} it follows that $\| \eta_{1/k} \ast E ( \phi w ) - E (\phi w) \|_{\Lp{1}{\R^{d}}}  \underset{k \rightarrow \infty}{\rightarrow} 0$ 
	and since $\| \psi w \|_{\Lp{\infty}{\Omega}} = \| \psi \|_{\Lpw{\infty}{w}{\Omega}}$ the statement is shown. 
	For $\Lpw{2}{w^2}{\Omega}$ one can follow the same construction and arrive at
	\begin{align*}
		| \inner[X,X^\ast]{\phi}{\psi_k - \psi} | \le &  \| \eta_{1/k} \ast E ( \phi w ) - E ( \phi w ) \|_{\Lp{2}{\R^{d}}} \| \psi w \|_{\Lp{2}{\Omega}}.			
	\end{align*} 
	From here, the statement again follows with \cite[Appendix, Theorem 7]{Eva98} and the fact that $\| \psi w \|_{\Lp{2}{\Omega}} = \| \psi \|_{\Lpw{2}{w^2}{\Omega}}$.
\end{proof}

	\subsection{Composition operators on $\Lpw{p}{w}{\Omega}$-spaces}
	\label{subsec:comp_lpw}
	
	In this section, we study the Koopman and the Perron-Frobenius operator on the space $\Lpw{p}{w}{\Omega}$. From now on, we fix the Banach space $X:=\Lpw{1}{w}{\Omega}$ and its dual space $X^*=\Lpw{\infty}{w}{\Omega}$ as well as the Hilbert space $H:=\Lpw{2}{w^2}{\Omega}$ which are related via the embeddings from \cref{lem:lpw_embeddings}. If a statement holds true in both of these spaces we will use $Y \in \{ X, H\}$ as a placeholder. Since the Koopman operator associated with the dynamics \eqref{eq:dyn_sys} is a special composition operator, we largely follow the (more general) exposition from \cite[Chapter 2]{SinM93}. For this purpose, we consider the measure space $(\Omega, \mathcal{B}, \mu)$ with measure
	\begin{equation*}
	\mu(B) := \int_{B} w(z) \dint z \qquad \text{for} \; B \in \mathcal{B}.
	\end{equation*}
	Note that $\mu$ is a $\sigma$-finite measure.
	Let $T\colon \Omega \to \Omega$ then be a \emph{measurable} transformation on $\Omega$, i.e., assume that $T^{-1}(S)\in \mathcal{B}$ for all $S\in \mathcal{B}$. Further assume that $T$ is \emph{non-singular} meaning that $\mu(S)=0$ implies $\mu(T^{-1}(S))=0$ for all $S\in \mathcal{B}$. By the Radon-Nikod\'{y}m theorem (\cite[Theorem 6.11]{Alt12}), there exists $\rho_T\in  \Lpw{1}{w}{\Omega}$ s.t. 
	\begin{equation}
	 \mu(T^{-1}(S)) = \int_S \rho_T(x) w(x) \dint x \qquad \forall S \in \mathcal{B}.
	 \label{eq:weighted_measure}
	\end{equation}
	Consequently, if $\| \rho_T \|_{\Lp{\infty}{\Omega}} < \infty$ then $T$ is non-singular. If $T\in C^1(\Omega,\Omega)$ then a change of variables (\cite[Theorem 7.26]{Rud87}) implies that 
	\begin{align}
	\label{eq:radon_nikodym_transformation}
		\rho_T(a) = \left\{ 
	\begin{array}{lcc}
		| \det \mathrm{D} T (z) |^{-1} \frac{w(a)}{w(z)}, & \quad \text{if there exists } \; z \in \Omega,\; \text{s.t.}\; a = T(z),\\
		0 & \quad \text{else}.
	\end{array}
	\right.	
	\end{align}
The non-singularity of $T$ assures that the composition operator
\begin{align*}
	\begin{array}{rlcl}
		C_T \colon & \Lpw{p}{w}{\Omega} & \to & \Lpw{p}{w}{\Omega}, \\
		& \varphi & \mapsto & \varphi \circ T,
	\end{array}
\end{align*}
is well-defined.
\begin{definition}
\label{defi:lyapunov:composition_semigroup}
	 On $Y^\ast$ with $ Y \in \{H,X\}$, we define the \emph{composition semigroup} as the family of composition operators with respect to the transformation $T = \Phi^t$ induced by the solution operator of \eqref{eq:dyn_sys} by
	\[
	\begin{array}{rlcl}
		S^\ast \colon & [0,\infty) & \to & \mathcal{L}( Y^\ast ) , \\
		& t & \mapsto & S^\ast(t)
	\end{array}
	\]
	and $ S^\ast(t) \varphi = \varphi \circ \Phi^t$ for all $\varphi \in Y^\ast$.
\end{definition}
To justify the name, we will show in \cref{theorem:lyapunov:semigroup_property} that $S^*(t)$ is a well-defined weak-* continuous semigroup, provided that the weight function $w$ and the dynamic $f$ satisfy additional properties.
\begin{assumption}\label{assumption:comp_weight}
	We assume that the weighting $w$ and the dynamic $f$ from equation \eqref{eq:dyn_sys} are such that
	\begin{enumerate}[label=(\roman*)]
		\item $f_i \in \Lpw{\infty}{w}{\Omega} \cap \C{1}{\bar{\Omega}}$, for $i = 1, \dots, n$
		\item 
		the following inequality holds
		\[
		\esssup_{x \in \Omega} \; - \frac{ f(x)^\top \nabla{w(x)} }{w(x)} = \omega_0 < \infty.
		\]
	\end{enumerate} 
\end{assumption}
From now on, we will always assume that \cref{assumption:comp_weight} is satisfied. The subsequent inequality will be used in various places throughout the rest of this paper. It provides an exponential bound for the weight along trajectories.
\begin{lemma}
	\label{lemma:decay_bound}
	Let $(\Phi^{t}(z))_{t \ge 0}$ denote the trajectory of the dynamical system \eqref{eq:dyn_sys}. It holds that
	\[
		w(z) \le \exp( t \omega_0 ) w(\Phi^t(z)) \qquad \text{for}\; t \in [0,\infty),\; z \in \Omega.
	\]
\end{lemma}
\begin{proof}
	For $t\ge s\ge0,$ differentiation of $w$ along trajectories yields 
	\begin{align*}
		\frac{\mathrm{d}}{\mathrm{d}s} w(\Phi^{t-s}(z) ) = - f(\Phi^{t-s}(z))^\top \nabla{w(\Phi^{t-s}(z))} \le \omega_0 w(\Phi^{t-s}(z)),
	\end{align*}
	where the last inequality follows from \cref{assumption:comp_weight}. The assertion now follows with Gronwall's lemma.
\end{proof}
\begin{theorem}
\label{theorem:lyapunov:semigroup_property}
	For $Y \in \{ X, H\}$ the composition semigroup $S^\ast(t) \in \mathcal{L}(Y^\ast)$ from  \cref{defi:lyapunov:composition_semigroup} is a well-defined weak-* continuous semigroup with infinitesimal generator
	\begin{align*}
		& A^\ast \colon \mathcal{D}(A^\ast) \subseteq Y^\ast \to Y^\ast,\quad A^\ast \varphi = f^\top \nabla \varphi.
	\end{align*}
	with $W^{1,\infty}(\Omega) \cap \Lpw{\infty}{w}{\Omega} \subseteq \mathcal{D}(A^\ast)$ in the case $Y = X$ and $W^{1,2}(\Omega) \cap \Lpw{2}{w^2}{\Omega} \subseteq \mathcal{D}(A^\ast)$ in the case $Y = H$.

\end{theorem}
\begin{proof}
	First we have to show that $S^\ast(t) \in \mathcal{L}(Y^\ast)$. We begin with the case $Y =H = \Lpw{2}{w^2}{\Omega}$.
	For this, consider the trajectory $(\Phi^t(z))_{t\ge0}$ as a mapping  
	\begin{align*}
	 \Phi\colon \mathbb R_+ \times \mathbb R^{d} \to \mathbb R^{d},\quad  \Phi(t,z)=\Phi^t(z).
	\end{align*}
	The dynamic of the system given by equation \eqref{eq:dyn_sys} then reads 
	\begin{align*}   
	 \frac{\partial }{\partial t}\Phi(t,z) = f(\Phi(t,z)).
	\end{align*}
	From \cite{Gron19}, we know that $J(t):=\mathrm{D}\Phi^t(z)$ exists and solves the linear ordinary differential equation
	\begin{align*}
		\left\{ \begin{array}{rcl}
			\frac{\mathrm{d}}{\mathrm{d}t} J(t) & = & \underbrace{ \left( \mathrm{D}f (\Phi^t(z))\right) }_{ =: A(t)} J(t) \qquad \text{for} \; t > 0\\
			J(0) & = & I
		\end{array}\right.
	\end{align*}
	and is thus given by $J(t) = \exp \left( \int_0^t A(s) \dint s \right).$ The properties of the matrix exponential imply
	$$
	  \det J(t) = \exp \left( \tr \left(  \int_0^t A(s) \dint s \right) \right).
	$$
	Since the system was assumed to be \emph{flow-invariant} it holds $\Phi^t( \bar{\Omega} )\subseteq \bar{\Omega}$. Furthermore, $f_i \in \C{1}{\bar{\Omega}}$ by  \Cref{assumption:comp_weight} and, hence, $\mathrm{D}f\in \C{0}{\bar{\Omega}}$, so that we conclude that $A(t)$ is bounded. As a consequence, there exist $\alpha(t)$ and $\beta(t)$ such that
	\begin{equation*}
		0 <  \alpha(t) \le \det \mathrm{D}\Phi^t(z)  \le \beta(t) < \infty.
		\label{eq:boundnessDx}
	\end{equation*}
	From \cite[Corollary 2.1.2]{SinM93}, for the norm of the composition operator $S^\ast(t)$ we find 
	\[
		\| S^\ast(t) \|_{\mathcal{L}( \Lpw{2}{w^2}{\Omega}) } = \| \rho_{\Phi^t(\cdot)} \|^{1/2}_{\Lp{\infty}{\Omega}},
	\]
	where $\rho_{\Phi^t(\cdot)}$ denotes the Radon-Nikod\'{y}m derivative.
	Similar to equation \eqref{eq:radon_nikodym_transformation}, $\rho_{\Phi^t(\cdot)}$ is determined by a change of variables \cite[Theorem 7.26]{Rud87} such that with \cref{lemma:decay_bound}, we obtain the bound
	\begin{align*}
		\| \rho_{\Phi^t(\cdot)} \|_{\Lp{\infty}{\Omega}} &= \esssup_{z \in \Phi^{-t}(\Omega)} \left| \frac{w^2(\Phi^{-t}(z))}{\det \mathrm{D} \Phi^t(z)  w^2(z)} \right|  =  \esssup_{z \in \Omega} \left| \det \mathrm{D} \Phi^t( \Phi^{t}(z))^{-1} \frac{w^2(z)}{w^2(\Phi^t(z))} \right| \\
		&\le \esssup_{z\in \Omega} \left| \det \mathrm{D} \Phi^t(z)^{-1} \right| \exp(2 \omega_0 t ) < \infty.
	\end{align*}
	Therefore the composition operator $S^\ast(t)$ is well-defined and bounded on $\Lpw{2}{w^2}{\Omega}$.
	Next, let us consider $Y =X= \Lpw{1}{w}{\Omega}$. Then $Y^\ast = X^\ast = \Lpw{\infty}{w}{\Omega}$ and we obtain that
	\begin{align*}
		\| S^*(t) \varphi \|_{\Lpw{\infty}{w}{\Omega}} &= \underset{z \in \Omega} { \esssup } \; | \varphi(\Phi^t(z)) w(z) | =  \underset{z \in \Omega} { \esssup } \left| \varphi(\Phi^t(z)) w( \Phi^t(z) ) \frac{w(z)}{w( \Phi^t(z) ) } \right| \\
		& \le  \| \phi \|_{\Lpw{\infty}{w}{\Omega}}\exp( t \omega_0 ) < \infty,
	\end{align*}
	where the inequality again follows from \cref{lemma:decay_bound}. Consequently, $S^\ast(t) \in \mathcal{L}(Y^\ast)$ for $Y \in\{H,X\}.$ Let us show that $S^\ast(t)$ is a weak-* continuous semigroup. Due to the time invariance of \eqref{eq:dyn_sys}, it immediately follows that $\Phi^t(\Phi^s(z)) = \Phi^{t+s}(z)$ and, consequently, $S^\ast(t+s) = S^\ast(t)S^\ast(s).$
	It remains to show weak-* continuity of $S^*(t).$	
	Let us consider $\varphi \in W^{1,\infty}( \Omega )$ first. 
	Then $\varphi$ is Lipschitz continuous and by the multidimensional mean value theorem \cite[Theorem 5.19]{Rud76} for some $\xi \in (0,t)$ it holds
	\begin{align*}
		& | \varphi( \Phi^t(z)) - \varphi ( z ) | w(z)  \le  L \| \Phi^t(z) - \Phi^0(z) \| w(z)
		 \le Lt \left\| \tfrac{\mathrm{d}}{\mathrm{d}t} \Phi^\xi(z) \right\| w(z) \\
		&=  L t \| f(  \Phi^\xi(z) ) \| w(z) 
		 \le   L t \| f( \Phi^\xi(z) ) w(\Phi^\xi(z))\|  \frac{w(z)}{w(\Phi^\xi(z))}
		\le L t \exp( \omega_0 t ) \|  f  \|_{\Lpw{\infty}{w}{\Omega}}.
	\end{align*}
	This obviously implies that $
	\lim\limits_{t \rightarrow 0} \| S^*(t) \varphi - \varphi \|_{Y^\ast} = 0
$	for $Y^\ast = X^*=\Lpw{\infty}{w}{\Omega}$. The case $Y =H = \Lpw{2}{w^2}{\Omega}$ follows similarly. By  \cref{lemma:w1inf_dense} we know that for $\varphi \in Y$ there exists $\varphi_n \in W^{1,\infty}(\Omega)$ such that $ 
	\varphi_n \overset{\ast}{\rightharpoonup} \varphi.$
	For arbitrary $\rho \in Y$ and $n \in \N$, consider
	\begin{align*}
		& \lim_{t \rightarrow 0} | \inner[Y,Y^\ast]{ \rho }{  S^\ast(t) \varphi - \varphi } | \\
		&\quad =   \lim_{t \rightarrow 0} | \inner[Y,Y^\ast]{ \rho }{  S^\ast(t) \varphi - \varphi  - ( S^\ast(t) \varphi_n - \varphi_n ) + (S^\ast(t) \varphi_n - \varphi_n) } |\\
		&\quad \le  \lim_{t \rightarrow 0} \left( \| \rho \|_Y \| S^\ast(t) \varphi_n - \varphi_n \|_{Y^\ast} + \left( 1 + \| S^\ast(t) \| \right) | \inner[Y,Y^\ast]{ \rho }{  \varphi_n - \varphi } | \right) \\
		&\quad \le   2 | \inner[Y,Y^\ast]{ \rho }{  \varphi_n - \varphi } |
	\end{align*}
	which proves that $S^\ast(t)$ is weak-* continuous. For the generator note that
	\[
		A^\ast \varphi = \lim_{t \rightarrow 0} \frac{S^\ast(t) \varphi - \varphi}{t} = \lim_{t \rightarrow 0} \frac{\varphi(\Phi^t(\cdot)) - \varphi}{t} = \frac{\mathrm{d}}{\mathrm{d}t} \varphi(\Phi^t(\cdot)) = f^\top \nabla \varphi
	\]
	where $\varphi \in \mathcal{D}(A^\ast)$ if and only if $f^\top \nabla \varphi \in Y^\ast$ exists in a weak sense. In the case $Y = X$ it is sufficient if $\varphi \in W^{1,\infty}(\Omega)$, since $f_i \in \Lpw{\infty}{w}{\Omega}$ by assumption. For $Y = H$ we  need $\varphi \in W^{1,2}(\Omega)$ instead.

\end{proof}
The result from \cite[Theorem 1.6]{EngN06} immediately yields that any weakly continuous semigroup is also strongly continuous. With the previous \Cref {theorem:lyapunov:semigroup_property} this means that $S(t)$ and in the case $Y = H = \Lpw{2}{w^2}{\Omega}$ also $S^\ast(t)$ defines a strongly continuous semigroup.

\begin{lemma}
	\label{lemma:lyapunov:adjoint_semigroup}
	If $ Y \in \{ X, H\}$ the preadjoint $S(t) \colon Y \to Y$ of the composition semigroup $S^\ast(t)$ 
	is given by
	\begin{align*}   	 
	( S(t) \rho )(a) &= \left\{ \begin{array}{ll}
		\mu(a) \rho(\Phi^{-t}(a)) \quad & \text{if } a \in \left( \Phi^t(\Omega) \right)^\circ, \\
		0 \qquad & \text{else}
	\end{array}\right. \\
	\mu(a) &:= \left| \det( \mathrm{D}\Phi^t(a)) \right|^{-1} \frac{w^2\left(\Phi^{-t}(a) \right) }{w^2(a)} \ge 0.
	\end{align*}
\end{lemma}
\begin{proof}
	By \cref{assumption:comp_weight} we know that $f_i \in \C{1}{\bar{\Omega}}$ and therefore $f_i$ is Lipschitz continuous. Since the tangent condition is also fulfilled the trajectories are uniquely determined \cite[Chapter III, Theorem XVIII (b) ]{Wal98}. Hence $\Phi^t \colon \bar{\Omega} \to \Phi^t(\bar{\Omega})$ is bijective and to any $a\in\Phi^t(\bar{\Omega}) \subseteq \bar{\Omega}$ we obtain a unique $\left( \Phi^{t} \right)^{-1}(a) = \Phi^{-t}(a) \in \bar{\Omega}$. 
	Now let $\rho \in X$ and $\varphi \in X^\ast$ then we can use a change of variables $ z := \Phi^{-t}(a)$ (\cite[Theorem 7.26]{Rud87}) to compute
	\begin{align*}
	\inner[X,X^\ast]{\rho}{S^{\ast}\varphi} 
		= & \int_\Omega \rho(z) \varphi(\Phi^t(z)) \; w^2(z) \dint z \\
		= & \int_{\left( \Phi^t(\Omega) \right)^\circ } \left| \det( \mathrm{D}\Phi^t(a)) \right|^{-1} \rho(\Phi^{-t}(a)) \varphi(a ) \; w^2(\Phi^{-t}(a)) \dint a \\
		= & \int_{\left(  \Phi^{t}(\Omega)\right)^\circ } \mu(a) \rho(\Phi^{-t}(a)) \varphi(a ) \; w^2(a) \dint a =  \inner[X,X^\ast]{ S(t) \rho }{\varphi}.
	\end{align*}
\end{proof}

\begin{proposition}
	\label{proposition:lyapunov:adjoint_generator}
	The generator $A \colon D(A) \subseteq \Lpw{1}{w}{\Omega} \to \Lpw{1}{w}{\Omega}$ of the semigroup $S(t)$ is given by
	\begin{equation*}
		\begin{array}{rccl}
			A \colon & D(A) \subseteq X & \to & X,\\
			& \phi & \mapsto & - \divergence \left( f \; \phi \right) - 2 \frac{ f(x)^\top \nabla w}{w} \phi
		\end{array}
	\end{equation*}
	It holds that
	\[
 	D := \{ \phi \colon \Omega \to \R \; \big| \; \| \phi \|_{\Lpw{1}{w}{\Omega}} + \sum_{ | \alpha | = 1 } \| \mathrm{D}^{(\alpha)} \phi \|_{\Lp{1}{\Omega}} < \infty \; \text{and} \; \phi  \,\rule[-3mm]{0.1mm}{4mm}_{I} \equiv 0 \} \subseteq \mathcal{D}(A) \subseteq \Lpw{1}{w}{\Omega}
	\]
	
	where $I := \{ x \in \partial \Omega \; \big| \; f(x)^\top \nu(x) \neq 0
	 \}$.
\end{proposition}
\begin{proof}
	Let $\phi \in D$ and $\psi \in \mathcal{D}(A^\ast) \subseteq \Lpw{\infty}{w}{\Omega}$. The divergence theorem yields
	\begin{align*}
		& \inner[X, X^\ast]{\phi}{A ^\ast \psi} 
		=  \int_{\Omega} \phi(x) f(x)^\top \nabla \psi(x) w^2(x) \dint x\\
		= & \int_{\partial \Omega} \underbrace{ \nu(x)^\top f(x)  \phi(x) }_{=0} \psi(x)  w^2(x)  \dint x 
		- \int_{\Omega} f(x)^\top \nabla \phi(x) \psi(x)  w^2(x) \dint x \\
		& \quad - \int_{\Omega} \frac{f(x)^\top \nabla w^2(x) }{w^2(x)} \phi(x) \psi(x) w^2(x) \dint x
		- \int_{\Omega} \divergence f(x) \phi(x) \psi(x) w^2(x) \dint x \\ 
		= & \inner[X,X^\ast]{ A \phi}{\psi }.
	\end{align*}
	 This implies that $A^\ast$ is the adjoint of $A$ and consequently $A$ is the preadjoint of $A^\ast$. By the statement given in \cite[Subsection 2.5]{EngN06} it follows that $A$ is the generator of $S(t)$.
\end{proof}
\begin{remark}
	Note that $A$ is the sum of a first order differential operator and a multiplication operator.
\end{remark}
For the computation of \eqref{eq:dyn_sys}, the asymptotic behavior of the semigroup for $t\to \infty$ is crucial. As it turns out, on the weighted space $\Lpw{p}{w}{\Omega}$, we obtain a simple condition for exponential stability.
\begin{theorem}
	If $f$ from \eqref{eq:dyn_sys} and $w$ satisfy
	\begin{equation}
			\esssup_{x \in \Omega} \; - \frac{ f(x)^\top \nabla{w(x)} }{w(x)} = \omega_0 < 0,
		\label{eq:l1_dissipative}
	\end{equation}
	then $S(t)$ is an exponentially stable semigroup of contractions of type $\omega_0$ over $X$.
	\label{thm:exp_stab}
\end{theorem}
\begin{proof}
	Since $\mu(a) \ge 0,$ with the explicit expression for $S(t)$ from  \cref{lemma:lyapunov:adjoint_semigroup}, we conclude that
	\begin{align*}
		| ( S(t) \phi )(a) | =& \left\{ \begin{array}{lcc}
			\mu(a) | \phi( \Phi^{-t}(a))	| & \qquad & \text{for} \; a \in \Phi^t(\Omega)\\
			0 & \qquad & \text{else}
		\end{array}\right.\\
		= & \; ( S(t) | \phi | )(a)
	\end{align*}
	for $a \in \Omega$. Consequently, this yields
	\begin{align}
		\| S(t) \phi \|_{\Lpw{1}{w}{\Omega}} & = \int_{\Omega} \left( S(t) | \phi | \right)(x) \;  w(x) \dint x \\
		& = \int_{\Omega} \left( S(t) | \phi | \right)(x) \underbrace{\frac{1}{w(x)}}_{ \in \Lpw{\infty}{w}{\Omega}} w^2(x) \dint x 
		 = \inner[X,X^\ast]{ | \phi | }{ S^\ast(t) \frac{1}{w(x)}}.
		 \label{eq:st_l1w}
	\end{align}
 Note that $w^{-1} \in \mathcal{D}(A^\ast)$ since
 \begin{equation}\label{eq:gen}
\begin{aligned}
\big \| \; A^\ast w^{-1} \; \big\|_{\Lpw{\infty}{w}{\Omega}} = & \esssup_{ x \in \Omega } \left| w(x) f(x)^\top \nabla \left( \frac{1}{w(x)} \right) \right| 
\le  \| f \|_{\Lpw{\infty}{w}{\Omega}} \big  \| w^{-1}(x) \big  \|_{W^{1,\infty}(\Omega)}.	
\end{aligned}
\end{equation}
Combining \eqref{eq:st_l1w} and \eqref{eq:gen} we conclude
	\begin{align*}
		&\frac{\mathrm{d}}{\mathrm{d}t} \| S(t) \phi \|_{\Lpw{1}{w}{\Omega}}  = \inner[X,X^\ast]{|\phi|}{S^\ast(t) A^\ast \frac{1}{w(x)}}
		= \inner[X,X^\ast]{S(t) | \phi | }{ - \frac{f(x)^\top \nabla w(x) }{w(x)^2}}\\
		&\quad = \int_\Omega \left( S(t) |\phi| \right)(x) \frac{-f(x)^\top \nabla w(x) }{w(x)} w(x) \dint x 
		 \le \omega_0 \int_\Omega \left( S(t) |\phi| \right) (x) w(x) \dint x \\
		 &\quad  = \omega_0 \| S(t) \phi \|_{\Lpw{1}{w}{\Omega}}.
	\end{align*}
	Gronwall's lemma implies $ 
	\| S(t) \phi \|_{\Lpw{1}{w}{\Omega}} \le \exp\left( \omega_0 t \right) \| \phi \|_{\Lpw{1}{w}{\Omega}}.$
\end{proof}
\begin{remark}
	Let us emphasize that the semigroup is not necessarily exponentially stable over $H$. For example consider, $f(x) := -x$ in $\Omega := B_1(0) \subseteq \R^{d}$ and $w(x) := \frac{1}{\|x\|}$. With regard to the assumptions of \cref{thm:exp_stab}, note that
	\[
	\esssup_{ x \in \Omega } - \frac{f(x)^\top \nabla w(x)}{w(x)} = \esssup_{ x \in \Omega } - \frac{x^\top x }{\|x\|^2 } = -1 < 0
	\]
	i.e., the lemma is applicable. However, for $\beta  > - \frac{{d}-2}{2}$ we can define functions
	\[
		u_\beta(x) := \|x\|^{\beta}
	\]
	which are elements of $H$. Indeed, observe that in spherical coordinates we have
	\begin{align*}
		\| u_\beta \|^2_{\Lpw{2}{w^2}{\Omega}} = \int_\Omega \| x \|^{2(\beta - 1)} \dint x = | S_1(0) | \int_0^1 r^{{d}-1+2(\beta - 1)} \dint r < \infty.
	\end{align*}
	However, for $x \in \Omega$ it also holds that
	\[
		(A^\ast u_\beta)(x) = -x^\top \nabla (\|x\|^\beta)  = - \beta 	x^\top x\|x\|^{\beta - 2} = - \beta u_\beta(x).
	\]
	This means that for ${d} \ge 3$ it follows that $u_\beta \in H$ for $\beta > -\frac{1}{2}$. Consequently, $\sigma( A^\ast ) \not \subseteq \mathbb{C}_{-}$ and since $\sigma( A ) = \overline{\sigma( A^\ast)}$ the semigroup cannot be exponentially stable over $H$.
\end{remark}
Many physical systems can be modeled with port-Hamiltonian systems for which we have a more specific characterization. 
\begin{proposition}
	\label{prop:port_hamiltionian}
	Suppose that $f(x)$ in \eqref{eq:dyn_sys} corresponds to a port-Hamiltonian system, i.e.,
	\[
	f(x) = ( J(x) - R(x) ) \nabla H(x) \qquad \forall x \in \Omega
	\]
	where
	\begin{enumerate}[label=(\roman*)]
		\item $H \colon \Omega \to \R_+$ is two times continuously differentiable with 
		\[ \| \nabla H(x) H^{-1/2}(x) \|_{\Lp{\infty}{\Omega} } < \infty. \]
		\item $R \colon \Omega \mapsto \R^{d \times d}$ is continuously differentiable and symmetric and positive semidefinite for all $x\in\Omega$.
		\item $J \colon \Omega \mapsto \R^{d \times d}$ is continuously differentiable and skew-symmetric for all $x\in\Omega$.
		\item The tangent condition $\inner{f(x)}{\nu(x)} \le 0$ for all $x \in \partial \Omega$ is fulfilled.
	\end{enumerate}
	If, in addition it holds that
	\begin{equation}
		\esssup_{ x \in \Omega } \left\{ - \frac{\nabla H(x)^\top R(x) \nabla H(x)}{ 2 H(x)}\right\} = \omega_0 < 0 \qquad \forall x \in \Omega
		\label{eq:pHsystem_stable}
	\end{equation}
	then $S(t)$ is an exponentially stable semigroup of contractions over the space $\Lpw{\infty}{w}{\Omega}$ of type $\omega_0$ with regard to the weighting $w(x) := H^{-1/2}(x)$. 
\end{proposition}
\begin{proof}
	\Cref{eq:l1_dissipative} and the assumptions of  \cref{theorem:lyapunov:semigroup_property} can be checked easily. 
\end{proof}
Note that condition $(i)$ and \eqref{eq:pHsystem_stable} are canonically satisfied if $R(x)$ is positive definite for every $x \in \Omega$ and the smallest eigenvalue can be bounded from below independently of $x \in \Omega$ and furthermore, it holds that
\[ 0 < \essinf_{x \in \Omega} \frac{ \| \nabla H(x) \| }{H(x)^{1/2}} \le \esssup_{x \in \Omega} \frac{ \| \nabla H(x) \| }{H(x)^{1/2}} < \infty. \]
In particular, the inequalities are true if $H$ is quadratic in the neighborhood of $\mathcal{M} := H^{-1}(\{0\})$.

	\section{Nuclear cost and sum of squares solution}	
	\label{sec:cost_operator_and_sum_of_squares}
As mentioned in the introduction, the structure of the cost $g$ plays a crucial role in the approximability of the cost function $v$. In particular, with $g$ and the underlying semigroup, we will derive an operator valued Lyapunov equation.
With the concept of nuclear operators in mind, we refer to a cost function $g$ as nuclear if it can be represented as a sum of squares of elements of the dual space. This class of cost functions is commonly found in many control problems.
\begin{definition}
	\label{defi:lyapunov:cost_operator}
	We say that the cost $g$ of the dynamical system from equation \eqref{eq:cost_func} is \emph{nuclear} with respect to $X$ if it can be represented as
	\[
	g(x) := \sum_{i=1}^\infty c_i(x)^2 \ \ \forall x \in \Omega, \ \ 
	\text{with} \ \ 
	c_i \in X^\ast \ \  \text{and} \ \  \sum_{i=1}^\infty \| c_i \|^2_{X^\ast} < \infty.
	\]
	In this case, we define the following observation operator $C$ and its adjoint $C^\ast$  
	\begin{align*}
		C\colon &X \to \ell_2,\quad \phi \mapsto \left( \inner[X,X^\ast]{\phi}{c_i} \right)_{i \in \N} \\
		C^*\colon& \ell_2 \to X^*,\quad   (a_i)_{i \in \N } \mapsto \sum_{i=1}^\infty a_i c_i .
	\end{align*}
\end{definition}
Note that for the particularly relevant case of a quadratic cost function, we may define $(\tilde{c}_1, \dots, \tilde{c}_r)^\top = \tilde{C} \in \R^{r \times {d}}$, and
$g(x) := x^\top \tilde{C}^\top \tilde{C} x = \sum_{i=1}^r  ( \underbrace{ \tilde{c}_i^\top x  }_{ =: c_i(x) } ) ^ 2$ for all $x \in \R^{d}.$
\begin{lemma}
\label{lemma:nuclear_infinite_time_admissible}
	If the semigroup $S(t)$ is exponentially stable of type $\omega_0$ over $X$ then
	\[
		\int_0^\infty \sum_{i=1}^\infty \| S^\ast(t) c_i \|_{X^\ast}^2 \dint t \le K^2.
	\]	
\end{lemma}
\begin{proof}
	Since $S(t)$ is exponentially stable, we find that 
	\begin{align*}
	\| S^\ast(t) c_i \|_{\Lpw{\infty}{w}{\Omega}} =& \sup_{\phi \in \Lpw{1}{w}{\Omega}, \| \phi \| > 0} \frac{ | \inner[\Lpw{1}{w}{\Omega},\Lpw{\infty}{w}{\Omega}]{ \phi }{S^\ast(t) c_i}|}{\| \phi\|_{\Lpw{1}{w}{\Omega}} } \\
		=& \sup_{\phi \in \Lpw{1}{w}{\Omega}, \| \phi \| > 0}   \frac{ | \inner[\Lpw{1}{w}{\Omega},\Lpw{\infty}{w}{\Omega}]{ S(t) \phi }{ c_i}|}{ \| \phi \|_{\Lpw{1}{w}{\Omega}} } 
		\le C \exp(w_0 t ) \| c_i \|_{\Lpw{\infty}{w}{\Omega}}.
	\end{align*}
	In view of \cref{defi:lyapunov:cost_operator}, we observe that
	\begin{align*}
		\int_0^\infty \sum_{i=1}^\infty \| S^\ast(t) c_i \|^2_{\Lpw{\infty}{w}{\Omega}} \dint t 
		\le C \int_0^\infty \exp(2 \omega_0 t )  \sum_{i=1}^\infty\| c_i \|_{\Lpw{\infty}{w}{\Omega}}^2 \dint t < \infty.
	\end{align*}
\end{proof}
 
The convergence of the integral allows us to define an operator $P$ by integrating
the cost following the composition semigroup along time in the space of nuclear operators thereby preserving the nuclearity. We will call the associated bilinear form value bilinear form as it will later on give rise to the Lyapunov function. More precisely, let us consider the following definition.
\begin{definition}
	For a given nuclear cost $g$ and its corresponding observation operator $C$, we define the \emph{value bilinear form} as 
	\begin{align}\label{defi:value_bilinear_form} 
	\inner[P]{\phi}{\psi} := \int_{0}^\infty \inner[\ell^2]{C S(t) \phi}{ C S(t) \psi } \dint t \qquad \forall \phi, \psi \in X.
	\end{align} 
\end{definition}
For an exponentially decaying semigroup over a Hilbert space it has already been shown (\cite{CurS01}) that for a finite rank observation operator $C$ the Gramian $P$ and the so-called observability map $\mathfrak{C}$ is a nuclear and a Hilbert-Schmidt operator, respectively. However, in our case we do not have an exponentially decaying semigroup over the entire Hilbert space $H$. Instead, we obtain the exponential decay only over $X$ and $X^\ast$, respectively. As it turns out, this is still sufficient because we assumed that the observation operator is bounded on $X$.
\begin{theorem}
	If $S(t)$ is exponentially stable over $X$, then the following holds: 
	\begin{enumerate}[label=(\roman*)]
		\item The observability map $\mathfrak{C} \colon H \to \Lp{2}{0,\infty; \ell_2 }$ with $\mathfrak{C}(\phi) = C S(\cdot) \phi$  for $\phi \in H$ is a Hilbert-Schmidt operator.
		\label{theorem:nuclear_gramian:enum:HS}
		\item The value bilinear form is bounded over $X$, i.e.,
		\[
			\inner[P]{\phi}{\psi} \le C \| \phi \|_{X} \| \psi \|_{X} \qquad \text{for} \; \phi, \psi \in X.
		\]
		\label{theorem:nuclear_gramian:enum:bounded}
		\item The value bilinear form admits the representation
		\[
		\inner[P]{\phi}{\psi} = \sum_{i=1}^\infty \inner[X,X^\ast]{\phi}{p_i} \inner[X,X^\ast]{\psi}{p_i} 
		\]
		with $p_i \in X^\ast$ satisfying
		$
			\sum_{i=1}^\infty \| p_i \|^2_H < \infty 
		$.
		\label{theorem:nuclear_gramian:enum:repres}
	\end{enumerate}  
	\label{theorem:nuclear_gramian}
\end{theorem}
\begin{proof}
	We start with \ref{theorem:nuclear_gramian:enum:HS}.
	Let $H_n$ be an orthonormal basis for $\Lp{2}{0,\infty}$. For $\phi \in X$ we define
	\begin{equation*}
		a_n^{(i)}(\phi) := \inner[\Lp{2}{0,\infty}]{H_n}{\inner[X,X^\ast]{\phi}{S^\ast(\cdot) c_i }} = \int_0^\infty H_n(t) \inner[X,X^\ast]{ \phi}{ S^\ast(t) c_i } \dint t .
	\end{equation*}
	We want to show that $a_n^{(i)} \in X^\ast$. Linearity is obvious.  For boundedness, we obtain 
	\begin{align*}
		 | a_n^{(i)}(\phi)  |^2 
		 \le \| H_n \|^2_{\Lp{2}{0,\infty}} \| \phi \|^2_X \int_0^\infty \| S^\ast(t) c_i \|^2_{X^\ast} \dint t \le K^2 \| \phi \|_X^2  
	\end{align*}
by the result from \cref{lemma:nuclear_infinite_time_admissible}. Furthermore, since $a_{n}^{(i)} \in X^\ast = \Lpw{\infty}{w}{\Omega} \subseteq \Lpw{2}{w^2}{\Omega} = H^\ast$ we obtain $a_n^{(i)} \in H^\ast$. Denoting by $e_j\in\ell_2$ the canonical unit vector, with the orthonormality of $H_n$ we can rewrite $\mathfrak{C} \phi$ for $\phi \in H$, i.e.
	\[
		\mathfrak{C} \phi = CS(\cdot)\phi= \left(\langle S(\cdot)\phi,c_i\rangle_{X,X^*} \right)_{i\in \mathbb N}= \sum_{i=1}^\infty \inner[H]{\phi}{S^\ast(\cdot) c_i} e_i = \sum_{i=1}^\infty \sum_{n=0}^\infty  a_n^{(i)}(\phi) H_n(\cdot) e_i.
	\]
	Since there exists a bijection between $\N \times \N_0$ and $\N$, for showing that $\mathfrak{C}$ is Hilbert-Schmidt it is sufficient to show $\sum_{i,n} \| a_n^{(i)} \|^2_{H^\ast} < \infty$. For this purpose, let $\{ \phi_k \}$ be an orthonormal basis of $H$. Using Parseval's indentity twice it follows that
	\begin{align*}
		\sum_{i=1,n=0}^\infty \| a_n^{(i)} \|_{H^\ast}^2 = & \sum_{i=1}^\infty \sum_{n=0}^\infty \sum_{k=1}^\infty | a_n^{i}(\phi_k) |^2 = 
		  \sum_{i,k=1}^\infty \hspace{-3px} \left( \sum_{n=0}^\infty  \inner[\Lp{2}{0,\infty}]{H_n}{\inner[X,X^\ast]{\phi_k}{S^\ast(\cdot) c_i }} ^2 \right)\\
		 = & \sum_{i,k=1}^\infty \| \inner[X,X^\ast]{\phi_k }{S^\ast(t) c_i} \|^2_{\Lp{2}{0,\infty}}.
	\end{align*}
	Using monotone convergence \cite[Theorem 4, Appendix E]{Eva98} to interchange summation and integration, we may rewrite this expression according to
	\begin{align*}
		\sum_{i=1,n=0}^\infty \hspace{-5px} \| a_n^{(i)} \|_{H^*}^2 = &  \sum_{i=1}^\infty \sum_{k=0}^\infty \int_0^\infty  \inner[X,X^\ast]{\phi_k }{S^\ast(t) c_i}^2 \dint t = \sum_{i=1}^\infty \int_0^\infty \sum_{k=0}^\infty \inner[X,X^\ast]{\phi_k }{S^\ast(t) c_i}^2 \dint t\\
		= & \sum_{i=1}^\infty \int_0^\infty \| S^\ast(t) c_i \|^2_H \dint t \le C(\Omega) \sum_{i=1}^\infty \int_0^\infty \| S^\ast(t) c_i \|^2_{X^\ast} \dint t < \infty.
	\end{align*}
	For \ref{theorem:nuclear_gramian:enum:bounded} we can directly use \cref{lemma:nuclear_infinite_time_admissible} and arrive at
	\begin{align*}	
		\inner[P]{\phi}{\psi} =& \int_0^\infty \inner[\ell_2]{C S(t) \phi}{C S(t) \psi } \dint t
		 = \int_0^\infty \sum_{i=1}^\infty \inner[X,X^\ast]{\phi}{S^\ast(t) c_i}\inner[X,X^\ast]{\psi}{S^\ast(t) c_i} \dint t \\
		 \le&  \int_0^\infty \sum_{i=1}^\infty \| S^\ast(t) c_i \|_{X^\ast}^2 \| \phi\|_X \|\psi \|_X \dint t \le K^2  \| \phi\|_X \|\psi \|_X.
	\end{align*}	
	Lastly, for \ref{theorem:nuclear_gramian:enum:repres} we note that there exists a representative $p_{i,n} \in X^\ast$, such that $a^{(i)}_n(\phi) = \inner[X,X^\ast]{\phi}{p_{i,n}}$. Therefore, by definition
	\begin{align*}
		\inner[P]{\phi}{\psi} &=  \inner[\Lp{2}{0,\infty;\ell_2}]{\mathfrak{C} \phi}{\mathfrak{C} \psi } 
		=  \sum_{i=1}^\infty\sum_{n=0}^\infty a_n^{(i)} (\phi ) a^{(i)}_n(\psi)\\
		&= \sum_{i=1}^\infty\sum_{n=0}^\infty \inner[X,X^\ast]{\phi}{p_{i,n}} \inner[X,X^\ast]{\psi}{p_{i,n}}.
	\end{align*}
	We have already shown that $\sum_{i=1}^\infty\sum_{n=0}^\infty \|a^{(i)}_n\|^2_{H^\ast} = \sum_{i=1}^\infty\sum_{n=0}^\infty \| p_{i,n} \|^2_H < \infty$ and because there exists a bijection between $\N \times \N_0$ and $\N$ the statement is proven.
\end{proof}
\begin{remark}
 It is important to note that the norm used in \ref{theorem:nuclear_gramian:enum:repres} is the norm of the Hilbert space $H$ and not the stronger norm of $X^\ast$.
\end{remark}

The existence of such a decomposition alone may already be useful, but if the decay is fast then it is justified to use efficient finite rank approximations, which are of great interest from a numerical point of view. For smooth enough $f \colon [0,\infty) \to \R$ a decay of the coefficients of the basis representation of the Laguerre polynomials can be shown under some assumptions \cite[Section 3]{GotO77}, by using the spectral properties of the Sturm-Liouville operator. This construction can also be applied to show that smooth dynamics and cost result in an eigenvalue decay that is faster than any polynomial. We note that an exponential decay rate in the slightly different setting where the semigroup is stable over $H$ has been shown \cite{Opm20}, which can likely be generalized to our setting. However, the following result also allows to treat dynamics and costs that only enjoy a Sobolev regularity $W^{m,\infty}(\Omega)$.

\begin{theorem}
	Let $C \colon X \to \R^r$ be of finite rank $r \in \N$ and $S(t)$ exponentially stable over $X$ with decay rate $\omega_0$. If  

%	\[
%		\range(C^\ast) \subseteq \bigcap_{k=0}^{m} \mathcal{D}((A^\ast)^k) \subseteq X^\ast \qquad \text{with $m$ even}
%	\]
	\[
	\range(C^\ast) \subseteq \mathcal{D}((A^\ast)^m) \subseteq X^\ast \qquad \text{with $m$ even}
	\] 	 
	
	then there exists $p_n \in X^\ast$ such that 
	\[
		 \inner[P]{\phi}{\psi} = \sum_{n=0}^\infty \inner[X,X^\ast]{\phi}{p_n} \inner[X,X^\ast]{\psi}{p_n} \quad \text{with} \quad \sum_{n=N}^\infty \|p_n\|^2_{H} \in \mathcal{O}(N^{-m}) . 
	\]
	\label{thm:spectral_decay}
\end{theorem}
\begin{proof}
	We construct the solution $\einner{P}$ similarly as the alternating direction implicit method from, e.g., \cite[ Remark 4.5]{OpmRW13} with shifts $\frac{1}{2}$. Let $L_n$ for $n \in \N$ be the normalized Laguerre polynomials \cite[Equation 5.1.1]{Sze39} \cite[Section 3]{GotO77}. We follow the construction from \cref{theorem:nuclear_gramian}, with the orthonormal basis of $\Lp{2}{0,\infty;\R}$ given as 
	\[
		H_n(t) := \exp(-t/2) L_n(t).
	\]	
	With these we will construct a sequence of decompositions, starting with 
	\begin{equation}
		a^{(i)}_{0,n}(\phi) := \int_0^\infty \inner[X,X^\ast]{\phi}{\exp(t/2)S^\ast(t) c_i} L_n(t) \exp(-t) \dint t.
		\label{eq:spectral_decay_dual}
	\end{equation}
	To construct the next decomposition $a^{(i)}_{1,n}$ we first note that normalized Laguerre polynomials $L_n$ are eigenvalues to the Sturm-Liouville problem \cite[Eq. 3.26]{GotO77} 
	\begin{equation*}
		\frac{\mathrm{d}}{\mathrm{d}t} \left(  p(t) \frac{\mathrm{d}L_n(t)}{\mathrm{d}t} \right) + (\lambda_n w(t) - q(t) ) L_n(t) = 0
	\end{equation*}
	with $p(t) = t \exp(-t), q(t) =0$ and $w(t) = \exp(-t)$. Therefore
	\begin{equation}
		\frac{\mathrm{d}}{\mathrm{d}t} \left(  t \exp(-t) \frac{\mathrm{d}L_n(t)}{\mathrm{d}t} \right) = - \lambda_n \exp(-t) L_n(t).
		\label{eq:spectral_decay_sturm_liouville}		
	\end{equation}
	 For the eigenvalues we find $\lambda_n = n$ \cite[Section 3, Page 42]{GotO77}. Therefore, we can substitute $L_n(t) \exp(-t)$ in  \eqref{eq:spectral_decay_dual} by the expression from \eqref{eq:spectral_decay_sturm_liouville} and obtain
	\begin{align}
			a^{(i)}_{0,n}(\phi) :=& - \frac{1}{\lambda_n} \int_0^\infty 	\inner[X,X^\ast]{\phi}{\exp(t/2)S^\ast(t) c_i} 	\frac{\mathrm{d}}{\mathrm{d}t} \left(  p(t) \frac{\mathrm{d}L_n(t)}{\mathrm{d}t} \right) \dint t.\\
			\intertext{Two times partial integration yields}
			a^{(i)}_{0,n}(\phi) =&\frac{1}{\lambda_n} \int_0^\infty h_1(t) L_n(t) \exp(-t) \dint t
			\label{eq:sturm_liouville_replacement}
	\end{align}
	with 
	\begin{align*}
		h_1(t) :=& \left( - \frac{\dint }{\dint t} \left( t \exp(-t) \frac{\dint }{\dint t}
		\inner[X,X^\ast]{\phi}{\exp(t/2)S^\ast(t) c_i}  \right) \right) \exp(t)\\
		= & 	\inner[X,X^\ast]{\phi}{\exp(t/2)S^\ast(t) \left( (t-1)\left( A^\ast + \frac{1}{2}I \right) c_i - t \left( A^\ast + \frac{1}{2}I \right)^2 c_i \right) } \\
		= & \sum_{k=0}^2 p_{1,k}(t) \inner[X,X^\ast]{\phi}{\exp(t/2)S^\ast(t) (A^\ast)^k c_i },
	\end{align*}
	where $p_{1,k}(t)$ are polynomials with degree smaller or equal than $1$. Note that $(A^\ast)^k c_i \in \Lpw{\infty}{w}{\Omega}$ and since $S^\ast(t)$ is exponentially stable, we conclude that
	\begin{align*}
		& \int_0^\infty h_1(t)^2 \exp(-t) \dint t \\
		 \le &  \int_0^\infty \sum_{k,k^\prime=0}^2 | p_{1,k}(t) p_{1,k^\prime}(t)  \inner[X,X^\ast]{\phi}{S^\ast(t) (A^\ast)^k c_i }\inner[X,X^\ast]{\phi}{S^\ast(t) (A^\ast)^{k^\prime} c_i }| \dint t \\
		 \le & \| \phi \|^2_X \int_0^\infty \sum_{k,k^\prime=0}^2 | p_{1,k}(t) p_{1,k^\prime}(t) | \| S^\ast(t) (A^\ast)^k c_i \|_{X^\ast}   \| S^\ast(t) (A^\ast)^{k^\prime} c_i \|_{X^\ast} \dint t\\
		 \le & \| \phi \|^2_X \int_0^\infty  \sum_{k,k^\prime=0}^2 C_k C_{k^\prime} | p_{1,k}(t) p_{1,k^\prime}(t) |   \exp( 2 \omega_0 t ) \dint t
		  \le C \| \phi \|^2_X.
	\end{align*}
	with $C_k := \sup_{t \in [0,\infty) }\| \exp( - \omega_0 t ) S^\ast(t) ( A^\ast )^k c_i \|_{X^\ast}$. Note that $C_k < \infty$ because $A^k c_i \in X^\ast$ and $S^\ast(t)$ is exponentially stable of type $\omega_0$ on $X^\ast$. Therefore $a^{(i)}_{1,n}(\phi) := \int_0^\infty h_1(t) L_n(t) \exp(-t) \dint t \in X^\ast = \Lpw{\infty}{w}{\Omega}$ with $\sum_{n=0}^\infty \| a^{(i)}_{1,n}\|^2_{H} < \infty$ by the same argument as in 
	\cref{theorem:nuclear_gramian}.
	We then can replace $L_n(t) w(t)$ in \eqref{eq:sturm_liouville_replacement} again and construct a new $h_2$ of the form
	\[
		h_2(t) = \sum_{k=0}^4 p_{2,k}(t) \inner[X,X^\ast]{\phi}{\exp(t/2)S^\ast(t) (A^\ast)^k c_i }
	\] with $a_{0,n}^{(i)}(\phi) = \frac{1}{\lambda_n^2} \int_0^\infty h_2(t) L_n(t) \exp(-t) \dint t$. This process can be repeated $m/2$-times and we obtain $a^{(i)}_{m/2,n}$ such that 
	\[
		a^{(i)}_{0,n}(\phi) = \frac{1}{\lambda_n^{m/2}} a^{(i)}_{m/2,n}(\phi)
		\quad \text{and} \quad 
		\sum_{n=0}^\infty \| a^{(i)}_{m/2,n}\|^2_{X^\ast} < \infty	.
	\]
	Since $\lambda_n = n$, this means that
	\[
			\sum_{n=N}^\infty \| a^{i}_{0,n} \|^2_{H} = \sum_{n=N}^\infty \left( n^{-\frac{m}{2}} \right)^2 \| a^{i}_{m/2,n} \|^2_{H} \le N^{-m} \sum_{n=1}^\infty \| a^{i}_{m/2,n} \|^2_{H} \in \mathcal{O}( N^{-m})
	\]
	and by the representation
	\[
		\inner[P]{\phi}{\psi} = \sum_{i=1}^r \sum_{n=0}^\infty \inner[X,X^\ast]{\phi}{a^{(i)}_{0,n}} \inner[X,X^\ast]{\psi}{a^{(i)}_{0,n}}
	\]
	from the proof of \cref{theorem:nuclear_gramian} the result follows.
\end{proof}
%\begin{remark}
%	This means the error between a finite rank approximation $\einner{P_N}$ of rank $N$ and $\einner{P}$ is of order $\mathcal{O}(N^{-m})$.
%\end{remark}
\begin{lemma}
	Let $f_i \in W^{ m - 1, \infty}(\Omega)\cap \Lpw{\infty}{w}{\Omega}$, $c_i \in W^{ m , \infty}(\Omega) \cap \Lpw{\infty}{w}{\Omega}$ with $m$ even and $c_i = 0$ for $i > r$ for some $r \in \N$. If the preadjoint $S(t)$ of the composition semigroup from \cref{defi:lyapunov:composition_semigroup} is exponentially stable of type $\omega_0$ over $X$, then 
	\[
		\sum_{i = N}^\infty \| p_i \|^2_{H} \in \mathcal{O}( N^{-m}) \qquad \text{for all} \; m \in \N
	\]
	for the representation of $\einner{P}$ from \cref{theorem:nuclear_gramian}.
	\label{thm:spectral_decay_cinfty}
\end{lemma}
\begin{proof}
	Let $\tilde{c} \in W^{ k , \infty}(\Omega)\cap \Lpw{\infty}{w}{\Omega}$ with $1\le k \le m$. Then $
		A^\ast \tilde{c} = f^\top \nabla \tilde{c} \;\in \; W^{ k -1 , \infty}(\Omega)$
	and furthermore
	\begin{align*}
		|(A^\ast \tilde{c} )(x)w(x)| = |f(x)^\top \nabla \tilde{c}(x) w(x)| \le \| f\|_{\Lpw{\infty}{w}{\Omega}} \| \nabla \tilde{c} \|_{\Lp{\infty}{\Omega}} < \infty.
	\end{align*}
	We conclude that $A^\ast \tilde{c} \in W^{k-1,\infty}(\Omega) \cap \Lpw{\infty}{w}{\Omega}$. Let us note that $(A^\ast)^0 c_i  = c_i \in \Lpw{\infty}{w}{\Omega}$ by assumption. Next set $\tilde{c} := \left( A^\ast \right)^{k-1} c_i$ for $1 \le k \le m$ and by recursion it follows
$		c_i \in \mathcal{D}((A^\ast)^m).
$	With \cref{thm:spectral_decay}, we obtain $
		\sum_{i=N}^\infty \| p_i \|_{H}^2 \in \mathcal{O}(N^{-m}).$
\end{proof}
\begin{remark}
This result indicates that the smoothness of $c_i$ should be compatible with the dynamics and that having more regularity is beneficial. Therefore, using $g_1(x) := \| x \|_2$ instead of $c_i(x) := x_i$ is a suboptimal choice, even though the observation operator $C$ has a lower rank of just one rather than $n$ and produces the same Lyapunov function.
\end{remark}

\begin{definition}
	We define the \emph{sum of squares solution} as
	\[
	v(x) := \sum_{i=1}^\infty p_i(x)^2  \qquad \text{for almost every $x \in \Omega$}
	\]
	where $p_i \in X^\ast$ are defined by the decomposition from \cref{theorem:nuclear_gramian}.
	\label{defi:sum_of_squares}
\end{definition}
Now we can show that the Lyapunov function can be recovered from the value bilinear form as a limit process using Dirac sequences.
\begin{theorem}
	If the preadjoint $S(t)$ of the composition semigroup from \cref{defi:lyapunov:composition_semigroup} is exponentially stable with rate $\omega_0$ then the Lyapunov function $v$ in \eqref{eq:cost_func} exists and it coincides with the sum of squares solution from 
	\cref{defi:sum_of_squares} almost everywhere. Furthermore, it holds $v \in \Lpw{\infty}{w^2}{\Omega}$.
	\label{theorem:sum_of_squares_is_normal}
\end{theorem}
\begin{proof}
	We use the standard mollifier $\eta_{\varepsilon}$ from \cite[Chapter 4.4]{Eva98} and define $\tilde{\eta}_{z,\varepsilon} := \frac{1}{w}\eta_{\varepsilon}(\cdot-z) \big|_{\Omega}$. Keep in mind that $\eta_{\varepsilon}$ is normalized w.r.t the $\Lp{1}{\R^{d}}$-norm,  i.e., $\| \eta_{\varepsilon} \|_{\Lp{1}{\R^{d}}} = 1$ for all $\varepsilon > 0$ and therefore $\| \tilde{\eta}_{z,\varepsilon} \|_{\Lpw{1}{w}{\Omega}} \le 1$. 
	
	Since $h w$ for any $h \in X^\ast = \Lpw{\infty}{w}{\Omega} $ can be extended to a locally integrable function on $\R^n$, we can use  \cite[Appendix C, Theorem 7]{Eva98} to conclude that
	\begin{equation}
		\lim_{\varepsilon \rightarrow 0} \inner[X,X^\ast]{ \tilde{\eta}_{z,\varepsilon}}{h} = h(z)w(z) \qquad \text{for almost every $z \in \Omega$}.	
		\label{eq:l1_convergence}
	\end{equation}
	With the result from \cref{lemma:nuclear_infinite_time_admissible} the following term is bounded independently of $\varepsilon$
	\begin{align*}
			\inner[P]{\tilde{\eta}_{z,\varepsilon}}{\tilde{\eta}_{z,\varepsilon}}  = & \int_0^\infty \sum_{i=1}^\infty \inner[X,X^\ast]{\tilde{\eta}_{z,\varepsilon}}{S^\ast(t) c_i }^2 \dint t
			 \le   \int_0^\infty \sum_{i=1}^\infty \| S^\ast(t) c_i \|_{X^\ast}^2 \dint t \le  K^2  < \infty.
	\end{align*}
	Now for almost every $z \in \Omega$ it holds
	\begin{align*}
		&\lim_{\varepsilon \rightarrow 0} \inner[P]{\tilde{\eta}_{z,\varepsilon}}{\tilde{\eta}_{z,\varepsilon}} 
		=  \lim_{\varepsilon \rightarrow 0} \int_0^\infty \sum_{i=1}^\infty \inner[X,X^\ast]{\tilde{\eta}_{z,\varepsilon}}{S^\ast(t) c_i}^2  \dint t.  \\
		\intertext{Let us use the dominated convergence theorem \cite[A.3.21]{Alt12} with the bound derived earlier and  \eqref{eq:l1_convergence} to conclude }
		\lim_{\varepsilon \rightarrow 0} \inner[P]{\tilde{\eta}_{z,\varepsilon}}{\tilde{\eta}_{z,\varepsilon}} & =   \int_0^\infty  \sum_{i=1}^\infty  \lim_{\varepsilon \rightarrow 0}  \inner[X,X^\ast]{\tilde{\eta}_{z,\varepsilon}}{S^\ast(t) c_i}^2   \dint t \\&  =  \int_0^\infty  \sum_{i=1}^\infty  c_i( \Phi^t(z))^2 w(z)^2  \dint t = v(z) w(z)^2.
	\end{align*}
	Now we will identify the limit by the sum of squares solution. We start by defining
	\begin{align*}
		q^{(\varepsilon)}_k \colon \Omega \to \R, \quad z \mapsto \inner[X,X^\ast]{\tilde{\eta}_{z,\varepsilon}}{p_k}
	\end{align*}
	and the extension 
	\[
			E \colon \Lp{s}{\Omega} \to \Lp{s}{\R^{d}}, \quad E \phi := \left\{ \begin{array}{lcl}
			\phi(x) & \qquad & x \in \Omega\\
			0 & \qquad & \text{else}
		\end{array}\right. 
	\]
	for any $1 \le s \le \infty$. We observe that for $z \in \Omega$
	\begin{align*}
		q^{(\varepsilon)}_k(z) = \int_{\Omega} \frac{1}{w(x)}\eta_\varepsilon(x - z ) p_k(x) w(x)^2 \dint x = \left( E \left( p_k w \right) \ast \eta_\varepsilon \right)(z).
	\end{align*}
	With Young's convolution inequality \cite[Section 4.13]{Alt12}, we can show 
	\begin{align*}
		\| q^{(\varepsilon)}_k \|_{\Lp{2}{\Omega}} \le \| \eta_{\varepsilon} \|_{\Lp{1}{\R^{d}}} \| p_k w \|_{\Lp{2}{\Omega}} = \| p_k \|_{\Lpw{2}{w^2}{\Omega}}
	\end{align*}
	and therefore for any $k \in \N$ and with the Hölder inequality \cite[Lemma 3.18]{Alt12} and the result from \cite[Theorem 7, Appendix C]{Eva98} we get
	\begin{align*}
		\| q^{(\varepsilon)}_k(\cdot)^2 - (p_k w)(\cdot)^2 \|_{\Lp{1}{\Omega}} = & \int_\Omega | q^{(\varepsilon)}_k(z)^2 - p_k(z)^2 w(z)^2 | \dint z\\
		= & \int_\Omega | ( q^{(\varepsilon)}_k(z) - p_k(z) w(z) )( q^{(\varepsilon)}_k(z) + p_k(z)w(z) ) |  \dint z\\
		\le & 2  \| q^{(\varepsilon)}_k - p_k w \|_{\Lp{2}{\Omega}} \| p_k \|_{\Lpw{2}{w^2}{\Omega}} \underset{\varepsilon \rightarrow 0}{\rightarrow} 0.
	\end{align*}
	We conclude for arbitrary $N \in \N$ that
	\begin{align*}
		  & \lim_{\varepsilon \rightarrow 0} \left\| \inner[P]{\tilde{\eta}_{\cdot,\varepsilon} }{\tilde{\eta}_{\cdot,\varepsilon} } - \sum_{k=1}^\infty (p_k w)(\cdot)^2 \right\|_{\Lp{1}{\Omega}}
		=  \lim_{\varepsilon \rightarrow 0} \left\| \sum_{k=1}^\infty q^{(\varepsilon)}_k(\cdot)^2  - \sum_{k=1}^\infty (p_k w)(\cdot)^2 \right\|_{\Lp{1}{\Omega}} \\
		\le & \lim_{\varepsilon \rightarrow 0}  \left( \sum_{k=1}^N  \left\| q^{(\varepsilon)}_k(\cdot)^2 - (p_k w)(\cdot)^2 \right\|_{\Lp{1}{\Omega}} +  \sum_{k=N+1}^\infty \|
		 q^{(\varepsilon)}_k \|^2_{\Lp{2}{\Omega}} +  \sum_{k=N+1}^\infty \| p_k \|^2_{\Lpw{2}{w}{\Omega}} \right)\\
		\le &   \sum_{k=1}^N  \lim_{\varepsilon \rightarrow 0} \left\| q^{(\varepsilon)}_k(\cdot)^2 - (p_k w)(\cdot)^2 \right\|_{\Lp{1}{\Omega}} +  2 \sum_{k=N+1}^\infty \| p_k \|^2_{\Lpw{2}{w}{\Omega}} 	=  2 \sum_{k=N+1}^\infty \| p_k \|^2_{\Lpw{2}{w}{\Omega}} 	 .
	\end{align*}
	Since $\sum_{k=1}^\infty \|p_k\|_{\Lpw{2}{w}{\Omega}}^2 < \infty$ and $N$ was arbitrary, it follows that 
	\[
	\lim_{\varepsilon \rightarrow 0} \| \inner[P]{\tilde{\eta}_{\cdot ,\varepsilon} }{\tilde{\eta}_{\cdot,\varepsilon} } - \sum_{k=1}^\infty (p_k w)(\cdot)^2 \|_{\Lp{1}{\Omega}} = 0.
	\] 
	Since $\infty > w(z) > 0$ almost everywhere, we obtain $v(z) = \sum_{i=1}^\infty p_i(z)^2$ almost everywhere on $\Omega$. The last inequality can be shown using the boundedness of $\einner{P}$
	\[
		\esssup_{z \in \Omega} w(z)^2 |v(z)| \le \esssup_{ z \in \Omega} \limsup_{\varepsilon \rightarrow 0} \inner[P]{\tilde{\eta}_{z,\varepsilon}}{ \tilde{\eta}_{z,\varepsilon} } \le  K^2 .
	\]
	with the bound from \cref{theorem:nuclear_gramian} \ref{theorem:nuclear_gramian:enum:bounded}. It follows $v \in \Lpw{\infty}{w^2}{\Omega}$.
\end{proof}
 
\begin{remark}
	Note that we do not need to assume that the limit for any of the trajectories $\lim\limits_{t \rightarrow \infty} \Phi^t(z)$ exists and that the sum of squares solution is independent of the decomposition we choose.
\end{remark}

	\section{An operator Lyapunov formulation}	
	\label{sec:op_lyap}
	
For linear systems with quadratic costs the Lyapunov function from \cref{eq:cost_func} is often computed by solving the algebraic Lyapunov equation \eqref{eq:lyap_fin_dim}. The following result will give a similar characterization for nonlinear systems by an infinite-dimensional operator Lyapunov equation.

\begin{theorem}\label{thm:operator_lyap}
	If the semigroup $S(t)$ is exponentially stable over $X$ then the value bilinear form from \cref{defi:value_bilinear_form} is the unique extension of the minimal solution of the operator Lyapunov equation over $H$
	\[
		\inner[P]{A \phi}{\psi } + \inner[P]{\phi}{ A \psi} + \inner[\ell^2]{C \phi}{ C \psi} = 0 \qquad \forall \phi,\psi \in \mathcal{D}(A) \subseteq H.
		\label{eq:op_lyapunov}
	\]
\end{theorem}
\begin{proof}
	From \cref{lemma:nuclear_infinite_time_admissible} and the embedding $X^\ast \subseteq H$ it follows that $C$ is infinite time admissible \cite[Definition 4.6.1]{TucW09} for $S(t)$ over $H$. By the result from \cite[Theorem 5.1.1]{TucW09} there exists a unique minimal solution $\einner{P_H} \colon H \times H \to \R$ that coincides with $\einner{P}$ on $H$. But $H$ is dense in $X$ by \cref{lem:lpw_embeddings} and $\einner{P_H}$ is also bounded w.r.t.~$\| \cdot \|_X$ by \cref{theorem:nuclear_gramian}. Therefore, by the continuous linear extension theorem \cite[E4.18]{Alt12} there exists a unique bounded linear extension to $X$, which is $\einner{P}$. 
	\label{lemma:op_lyapunov}
\end{proof}
In \cref{theorem:sum_of_squares_is_normal} we showed that the Lyapunov function $v$ exists if the semigroup is exponentially stable. If we assume that the Lyapunov function exists and satisfies some additional assumptions and the dynamic is dominated by a stable linear term around the origin, we obtain a converse implication.
\begin{proposition}
\label{proposition:exponential_decay_stable_in_origin}
	Let $f(x) = A x + \tilde{f}(x)$ be the dynamic of the system. Let us assume that the following conditions are fulfilled:
	\begin{enumerate}[label=(\roman*)]
		\item $\Real(\lambda_i(A)) < 0$ for all eigenvalues $\lambda_i(A)$ of $A$.
		\label{cond:negative_real_part}
		\item $f$ fulfills the tangent condition \cref{eq:tan_con}, i.e., $f(x)^\top  \nu(x) \le 0$ for all $x \in \partial \Omega$. 
		\item The Lyapunov function to the cost $g(x) := \| x \|^2$ exists and satisfies
		\[
		v( z ) := \int_{0}^\infty \| \Phi^t(z) \|^2 \dint t <  \infty \quad \text{for all} \; z \in \bar{\Omega}, \quad v^{1/2} \in W^{1,\infty}(\Omega).
		\]
	\end{enumerate}
	Then  $S^\ast(t) \colon \Lpw{\infty}{w}{\Omega} \to \Lpw{\infty}{w}{\Omega}$ is exponentially stable w.r.t.~$w(x) := \frac{1}{\|x\|}$.
\end{proposition}
\begin{proof}
	By assumption \ref{cond:negative_real_part} there exists a positive definite matrix $X$ such that
	\[
		A^\top X + X A + I_{d \times d} = 0.
	\]
	In the following step, we verify that the assumptions of \cref{thm:exp_stab} are fulfilled for the weighting $\tilde{w}(x) := \|x \|^{-1}_X$ with $\|x\|_X:=\sqrt{x^\top X x}$ and the domain $\Omega := B_r(0) \subseteq \Omega$ for some $r > 0$ small enough. For this purpose, we will focus on
	\begin{align*}
		 & - \frac{f(x)^\top \nabla \tilde{w}(x) }{ \tilde{w}(x) } = \frac{ x^\top A^\top X x + \tilde{f}(x)^\top X x}{2 \| x \|^2_X }
	\end{align*}
	where we used $\nabla \| x \|^{-1}_X = -\frac{Xx}{\|x\|^{3/2}_X}$ for $x \neq 0$. First let us consider the nonlinear part. Let $X^{1/2}$ be the matrix square root of $X$, then it holds
	\begin{equation}
		\frac{| \tilde{f}(x)^\top X x| }{ 2 \|x \|_X^2} \le \frac{\|X^{1/2}\tilde{f}(x)\|_2 }{\|x\|_X} \frac{\| X^{1/2} x \|_2 } {2 \|x \|_X} \le \frac{\lambda_{\text{max}}(X)}{2 \lambda_{\text{min}}(X)} \frac{\|\tilde{f}(x)\|_2}{\|x\|_2}\in \mathcal{O}(\|x\|_2).
		\label{eq:non_linear_part}
	\end{equation}
	Since $X$ solves the algebraic Lyapunov equation, for the linear part we obtain
	\begin{align}
	\frac{ x^\top A^\top X x}{2 \| x \|^2_X }	= & \frac{ x^\top \left( A^\top X + X A \right) x }{ 4 x^\top X x } \le \frac{ - \| x\|^2 }{4 x^\top X x } 
		\le - \frac{1}{4\lambda_{\text{max}}(X)} < 0.
		\label{eq:linear_part}
	\end{align}
	By combining \cref{eq:non_linear_part} and \cref{eq:linear_part}  we can choose $C,r > 0$ such that
	\[
		 \esssup_{x \in B_r(0)} \; - \frac{f(x)^\top \nabla \| x\|_X }{\|x \|_X } \le -C \lambda_{\text{max}}(X)^{-1} \quad \text{and} \quad f(x)^\top \nu(x) \le 0 \; \text{for} \; x \in \partial B_r(0) 
	\]
	where $\nu(x) = \frac{x}{\|x\|}$. Therefore the assumptions of \cref{thm:exp_stab} are fulfilled and with  \cref{theorem:sum_of_squares_is_normal} the sum of squares solution coincides with $v \big|_{B_r(0)}$. In fact, we even have $v\big|_{B_r(0)} \in \Lpw{\infty}{\tilde{w}^2}{B_r(0)}$ which implies  
	\[
		\esssup_{x \in B_r(0) } \frac{v(x)}{\|x\|^2} \le \esssup_{x \in B_r(0)}\frac{ \lambda_{\text{max}}(X)  v(x)}{ \|x\|^2_X} < \infty.
	\]
	On the other hand for $x \in \Omega \setminus B_r(0)$ the expression $\|x\|^2$ is bounded from below and $\esssup_{x \in \Omega} v(x) < \infty$ by assumption such that 
	\begin{equation}
		\esssup_{x \in\Omega} \frac{v(x)}{\|x\|^2}  < \infty.
		\label{eq:bound_v_norm}
	\end{equation}
	If we define $w(x) := v(x)^{-1/2}$, then with \cref{eq:lyap_pde} it holds that
	\[
		\esssup_{x \in \Omega} -\frac{f(x)^\top \nabla w(x)}{w(x)} = \esssup_{x \in \Omega} \frac{1}{2}\underbrace{ f(x)^\top \nabla v(x)}_{ = -\| x \| ^2 } v(x)^{-1} = 
			\esssup_{x \in \Omega} -\frac{1}{2}\frac{\| xf \|^2 }{v(x) } < 0.
	\]
	From the tangent condition $f(x)^\top \nu(x) \le 0$ and with \cref{thm:exp_stab} the semigroup $S(t) \colon \Lpw{1}{w}{\Omega} \to \Lpw{1}{w}{\Omega}$ is exponentially stable, i.e.,
	\[
		\| S(t) \phi \|_{\Lpw{1}{w}{\Omega}} \le C \exp( \omega_0 t )\| \phi\|_{\Lpw{1}{w}{\Omega}} \quad \text{for some} \; \omega_0 < 0.
	\]
	This however means that $S(t) \colon \Lpw{1}{\bar{w}}{\Omega} \to \Lpw{1}{\tilde{w}}{\Omega}$ with $\bar{w}(x) := \frac{1}{\|x\|}$ is also exponentially stable  since we can bound the norm w.r.t $\bar{w}$ by the norm w.r.t $w$ via 
	\[	
		\| \phi \|_{\Lpw{1}{\bar{w}}{\Omega}} = \int_{\Omega} \frac{| \phi(x) |}{\|x\|} \dint x \le \underbrace{\esssup_{x \in \Omega} \frac{v^{1/2}(x)}{\|x\|}}_{ < \infty \; \text{by} \; \eqref{eq:bound_v_norm}} \int_{\Omega} |\phi(x)| v^{-1/2}(x) \dint x = \tilde{C} \| \phi\|_{\Lpw{1}{w}{\Omega}}.
	\]    
\end{proof}

	\section{Numerical proof of concept}
	\label{sec:numerics}
	
	In this section, we briefly validate our theoretical findings by two small-scale numerical examples. Let us emphasize that our purpose is to demonstrate the potential of the rapidly decaying eigenvalues of the solution of the resulting matrix Lyapunov equation for numerical methods. In particular, we believe that it could establish a way for efficient tensor-based low rank solvers for large-scale Lyapunov functions, e.g., arising throughout the policy iteration for optimal feedback computations. A detailed treatise is however out of the scope of this manuscript and is subject of ongoing research.
	
	Here, we restrict ourselves to a simple two-dimensional setup based on a polynomial tensor basis and a straightforward discretization. In more detail, the discretization relies on Legendre polynomials or splines that are orthonormalized with respect to the $\Lpw{2}{w^2}{\Omega}$ norm using Gauss-Legendre quadrature and an eigendecomposition. The infinitesimal generator is discretized as a matrix and the resulting algebraic Lyapunov equation was solved using the  built-in method \verb|solve_continuous_lyapunov| from Scipy. The implementation can be downloaded\footnote{\url{https://git.tu-berlin.de/bhoeveler/koopman-based-operator-lyapunov}} and was done using Python version 3.9.15, TensorFlow version 2.11.0, Scipy version 1.8.1, and Numpy version 1.22.4. All simulations were conducted on a desktop computer equipped with an AMD R9 3900X processor, 64 GB of RAM and a Radeon VII graphics card.
	
	\subsection{A linear quadratic problem}
	
	We begin with linear (dissipative) dynamics and a quadratic cost function  over the domain $\Omega = [-1,1]^2$, i.e.,
	\[
		f(x_1,x_2) :=A_m \begin{pmatrix} x_1\\ x_2\end{pmatrix}= \begin{pmatrix}
			-2 & 1\\
			-1 & -3
		\end{pmatrix}\begin{pmatrix} x_1\\ x_2\end{pmatrix}  \quad \text{and} \quad g(x) := c_1(x)^2 + c_2(x)^2
	\]
	 with $c_i(x) := x_i$. The weighting is chosen as $w(x) := \| x \|^{-1}$. With regard to the compatibility of $f$ and $w$, note that the tangent condition is fulfilled, $f_i \in \Lpw{\infty}{w}{\Omega}$ and furthermore
	\[
		\omega_0 := \sup_{x} -\frac{f(x)^\top \nabla w(x)}{w(x)} = \sup_{x} \frac{x^\top A_m x }{\|x\|^2} = \lambda_{\text{max}} \left( \frac{1}{2}( A_m + A_m^\top ) \right) < 0.
	\]
	  By \cref{thm:exp_stab} the corresponding semigroup is an exponentially stable semigroup of contractions and furthermore, the assumptions of \cref{thm:spectral_decay_cinfty} are fulfilled for all $m \in \N$ leading to a super-polynomial decay.
% For the observation we used $c_i(x) := x_i$ and $C := \sum_{i=1}^2 \inner[X,X^\ast]{\cdot}{c_i}$. 
	In this specific case, it is easy to show that the eigenfunctions $p_i$ of $P$ are linear and that their representation as elements of $\R^2$ is a decomposition of the solution $X$ to the algebraic Lyapunov equation. In other words, one can show that the solution $P$ to the operator Lyapunov equation is of finite rank of at most $n=2$. This theoretical result is numerically confirmed by the eigenfunctions $p_i$ and eigenvalues in \cref{fig:linear_eigenfunctions} of $P$ where only the two largest eigenvalues are (numerically) non-zero and both of them correspond to linear eigenfunctions. \Cref{fig:linear_lyapunov_function} shows that the error between the calculated sum of squares solution and the reference solution, obtained by solving the matrix-valued Lyapunov equation, is approximately $10^{-13}$. The spiking behavior of the error in the corners of the domain seems to be caused by numerical instabilities of the Legendre polynomials which had a degree of up to 11 in this case.
	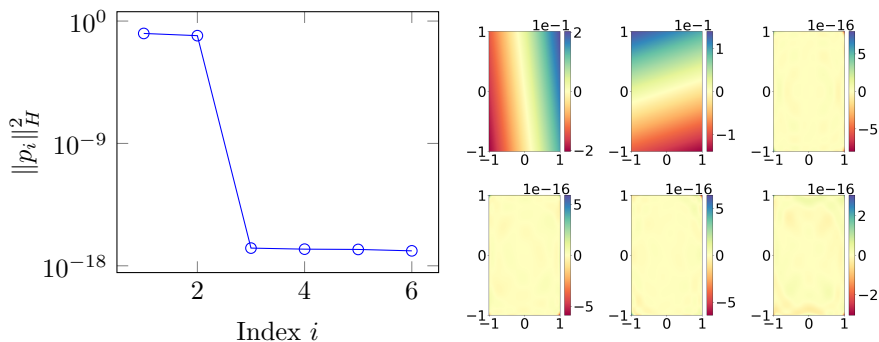
\begin{figure}[h]
		\centering
		\begin{minipage}{.45 \textwidth}
		\begin{tikzpicture}
			\begin{semilogyaxis}[    
					xlabel=Index $i$,    
					ylabel=$\| p_i \|^2_{H}$,					
				    legend pos=north east,
				    width=\textwidth,
				]
							
				\addplot[x=x, y=y, color=blue,mark=o, mark options={draw=blue,fill=white}] table  {
						x y 
						1 1.22804160e-01 
						2 8.43386971e-02 
						3 2.00059334e-17 
						4 1.67252373e-17 
						5 1.59431207e-17 
						6 1.25826413e-17 
				};											
%				\legend{Legend}
			\end{semilogyaxis} 
		\end{tikzpicture}
	\end{minipage}
	\begin{minipage}{.45 \textwidth}
		\includegraphics[width = \textwidth]{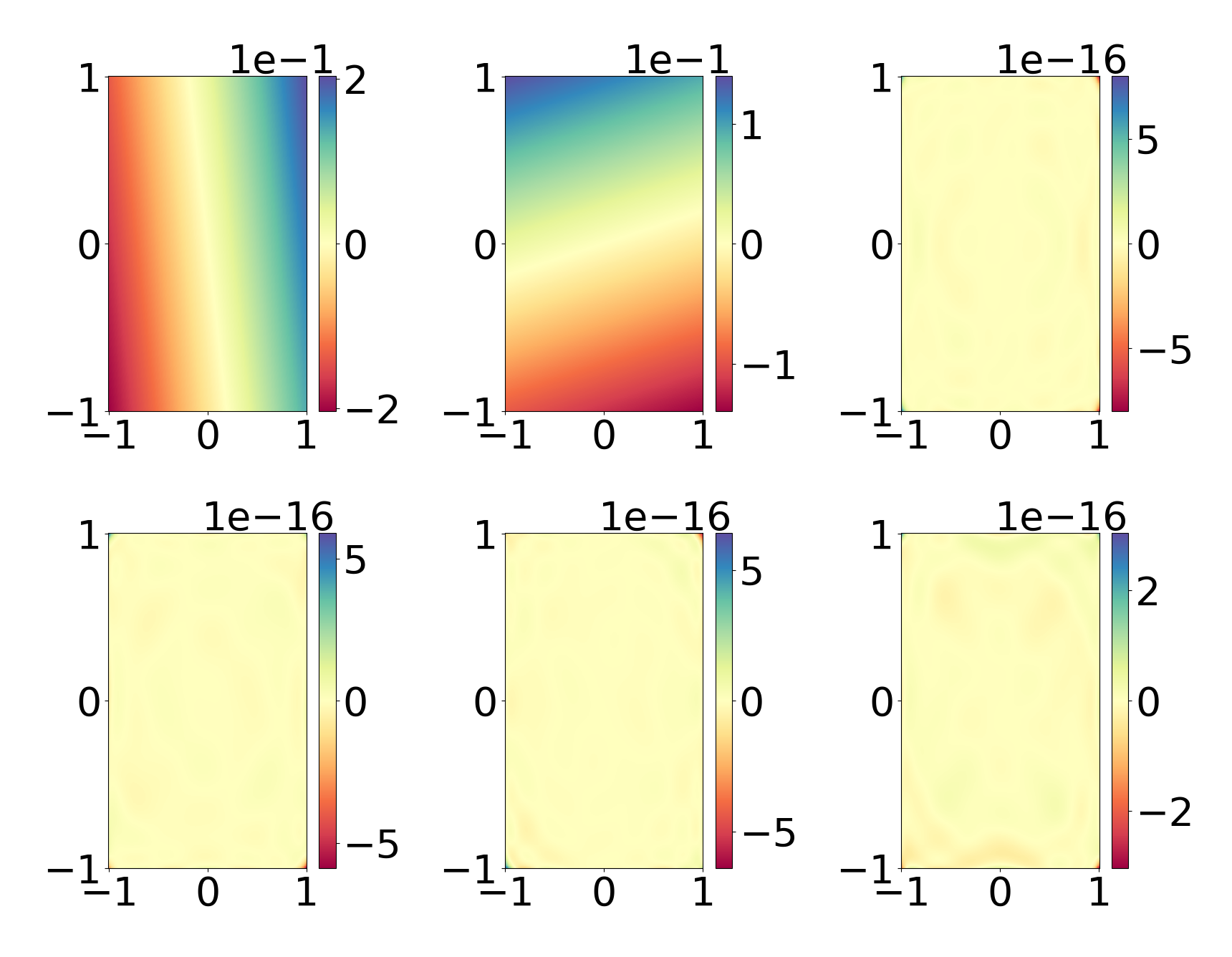}
	\end{minipage} 
		\caption{Left: The squared norm of the first six eigenfunctions $p_i$ of the linear example. Right: The first six eigenfunctions $p_i$.}
		\label{fig:linear_eigenfunctions}
	\end{figure}
	\begin{figure}[h]
		\includegraphics[width = .45 \textwidth]{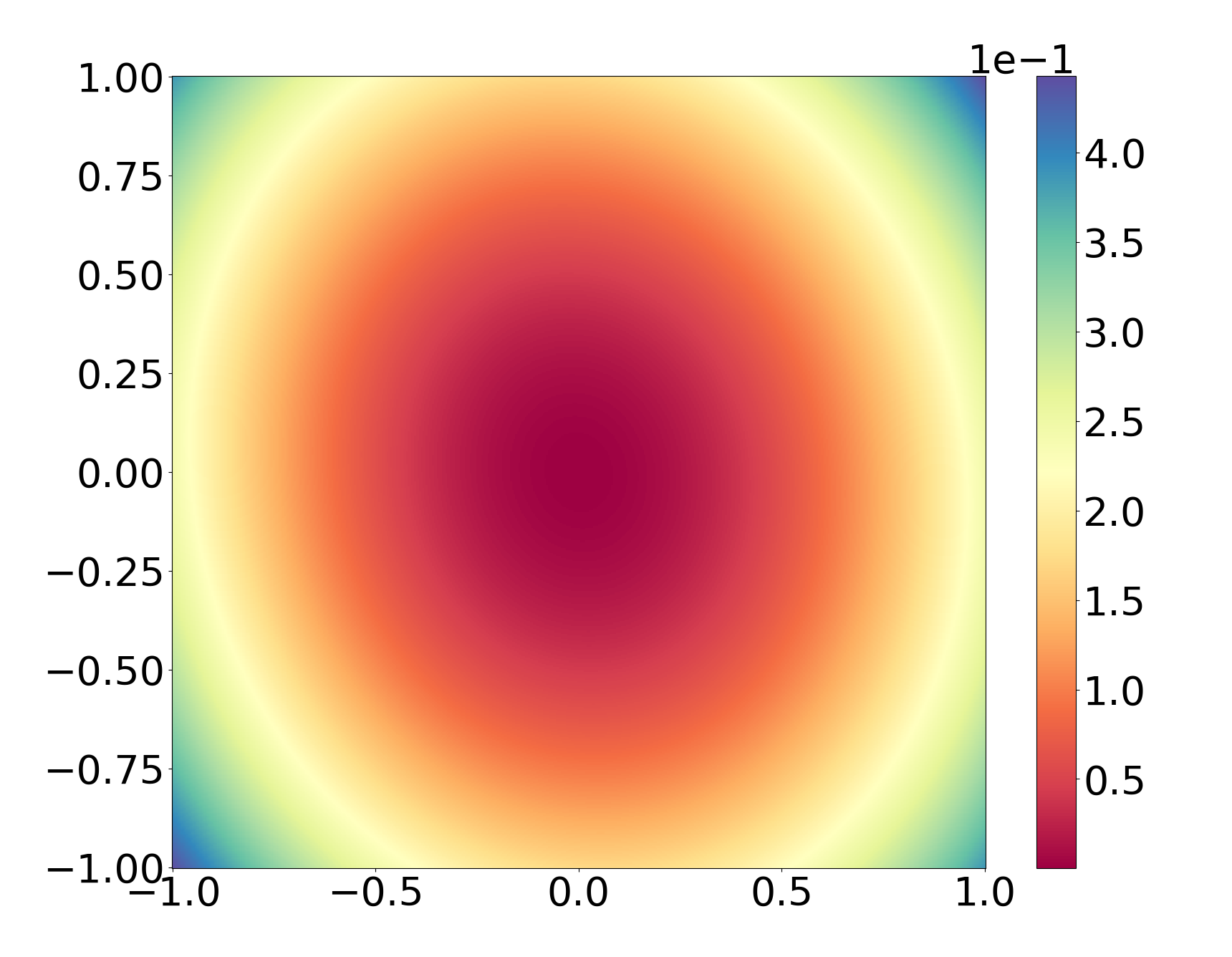}		
		\includegraphics[width = .45 \textwidth]{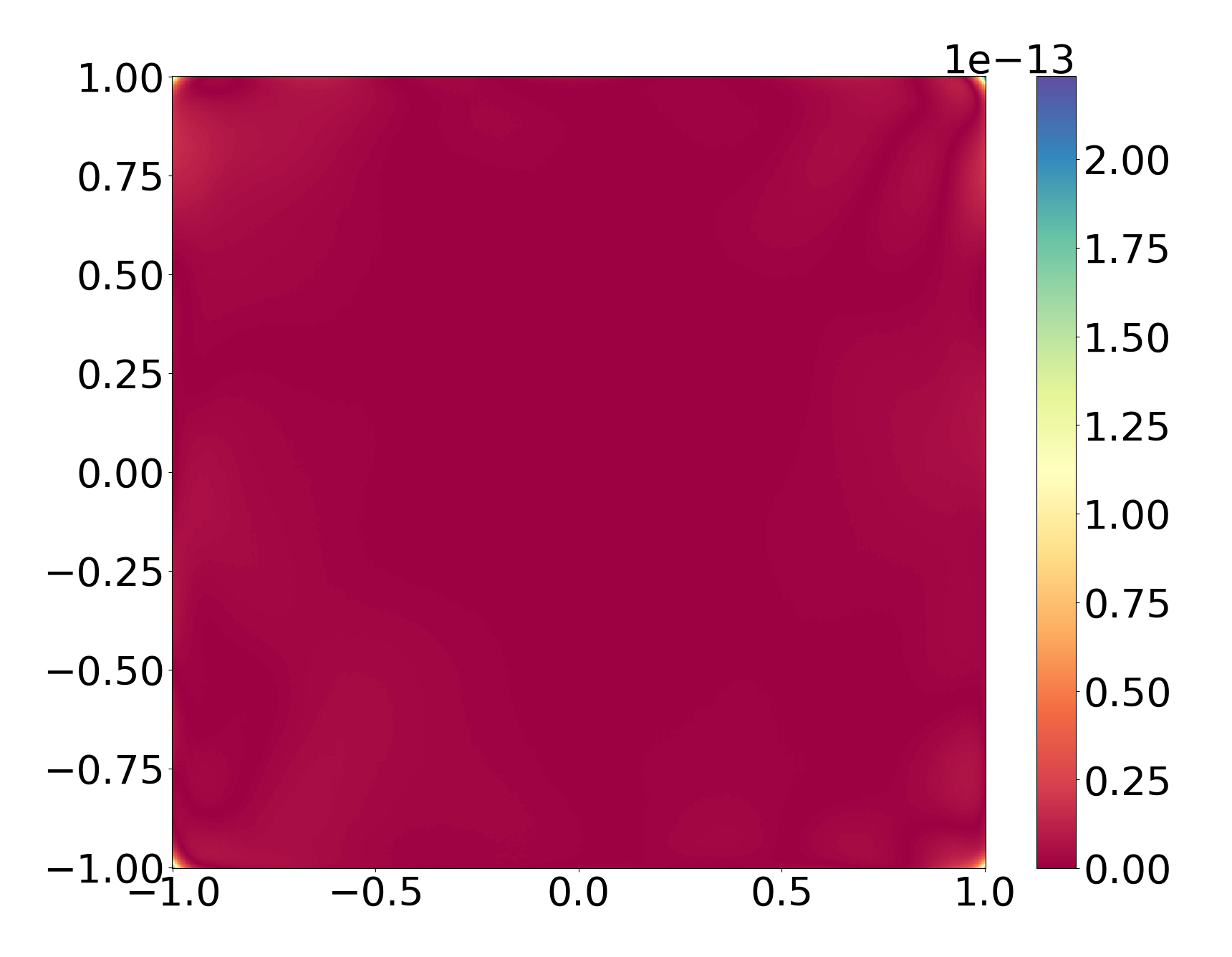}
		\caption{Left: The computed sum of squares solution for the linear example. Right: The error between computed solution and reference solution obtained by solving the Lyapunov matrix equation.} 
		\label{fig:linear_lyapunov_function}
	\end{figure}
	\subsection{Modified Van der Pol Oscillator}
	The Van der Pol oscillator is a common test example for nonlinear dynamics, see, e.g., \cite{AzmKK21}. While it is possible to use an appropriate weighting $w$ with $w\big|_{\mathcal{M}} = \infty$ to handle the undamped case (where $\mathcal{M}$ denotes the stable manifold), we include a friction term to create a dynamic with zero as the only accumulation point and choose $w(x) := \frac{1}{\|x\|}$ as in the linear quadratic case. To satisfy the tangent condition $f(x)^\top \nu(x) \le 0$, we add an additional term $x_1^3$. We consider the domain $\Omega := [-3,3] \times [-3,3]$ and examine the modified, damped Van der Pol oscillator dynamic and a simple quadratic cost
	\begin{align*}
		f(x_1,x_2) := \begin{pmatrix}
			x_2 - \alpha x_1^3\\
			- \mu ( x_1^2 -1) x_2 - x_1 - \eta x_2
		\end{pmatrix}
	\quad \text{and} \quad g(x_1,x_2) := c_1(x)^2 + c_2(x)^2
	\end{align*}
	with $\mu = 2$, friction term $\eta = 2.2$ and $\alpha = 1.5 \times 10^{-1}$. For the observation we again choose $c_i(x) := x_i$.
	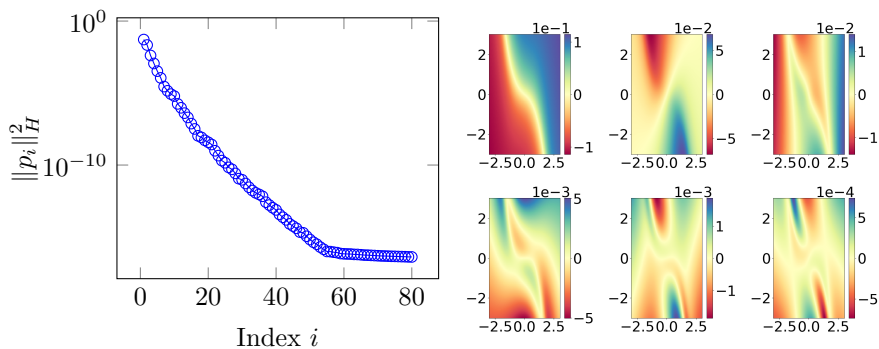
\begin{figure}[h]
	\centering
	\begin{minipage}{.45 \textwidth}
		\begin{tikzpicture}
			\begin{semilogyaxis}[    
				xlabel=Index $i$,    
				ylabel=$\| p_i \|^2_{H}$,
				legend pos=north east,
				width=\textwidth,
				]
				
				\addplot[x=x, y=y, mark=o, color = blue, mark options={draw=blue,fill=white,scale =1.} ] table{
					x y 
					1 0.05198775717388165
					2 0.021656555439024214
					3 0.004077479910617511
					4 0.0011344623949942345
					5 0.00033143534508742244
					6 0.000112855594377326
					7 2.814806943370669e-05
					8 1.4327217403970211e-05
					9 8.469107906964432e-06
					10 6.151224330175691e-06
					11 1.767560525068983e-06
					12 9.448866074614081e-07
					13 4.0254683244311614e-07
					14 2.0252915506545813e-07
					15 7.552196051663117e-08
					16 3.2191087955776954e-08
					17 1.0732757247179738e-08
					18 8.282826649802481e-09
					19 5.209460144272769e-09
					20 3.828698195294888e-09
					21 2.674504130738594e-09
					22 9.812242171037708e-10
					23 4.517587406497899e-10
					24 2.0420447015501133e-10
					25 1.4375968825552682e-10
					26 6.91709962490349e-11
					27 5.029809997937986e-11
					28 2.631588426496285e-11
					29 1.1403177078328388e-11
					30 9.46047802996907e-12
					31 5.237587099874618e-12
					32 2.8731395209821117e-12
					33 1.7078052703151237e-12
					34 1.1686552528845988e-12
					35 9.120497203379038e-13
					36 6.72436308431616e-13
					37 2.630797984179827e-13
					38 1.6654582307747675e-13
					39 1.0807219964510622e-13
					40 7.663394538499934e-14
					41 3.741603993310947e-14
					42 2.3661873924759697e-14
					43 1.5643242494478097e-14
					44 8.19066374981227e-15
					45 5.746278642616807e-15
					46 4.022150979942336e-15
					47 2.211604031577383e-15
					48 1.962211782398898e-15
					49 1.1297762608082232e-15
					50 6.721292706843092e-16
					51 4.360956149250667e-16
					52 3.159018959174429e-16
					53 2.0170367159863186e-16
					54 1.4174523998107107e-16
					55 9.97614830234511e-17
					56 9.676426260197884e-17
					57 8.869333598333572e-17
					58 8.242182843765623e-17
					59 7.14698896642597e-17
					60 6.812773454881852e-17
					61 6.673128047332937e-17
					62 6.585754154382771e-17
					63 6.246775857734556e-17
					64 6.157838722709983e-17
					65 5.881221541574059e-17
					66 5.590616496707927e-17
					67 5.512104239118911e-17
					68 5.3396571790524266e-17
					69 5.186584847515288e-17
					70 5.0221478347449806e-17
					71 4.877821396787599e-17
					72 4.80855181821911e-17
					73 4.713585407075262e-17
					74 4.600900452223406e-17 
					75 4.5500667844357873e-17
					76 4.4412082074407364e-17
					77 4.372627653229889e-17
					78 4.306239066647485e-17
					79 4.1945003693416754e-17
					80 4.146410741650374e-17 
				};											
				%				\legend{Legend}
			\end{semilogyaxis}
		\end{tikzpicture}
	\end{minipage}
	\begin{minipage}{.45 \textwidth}
		\includegraphics[width = \textwidth]{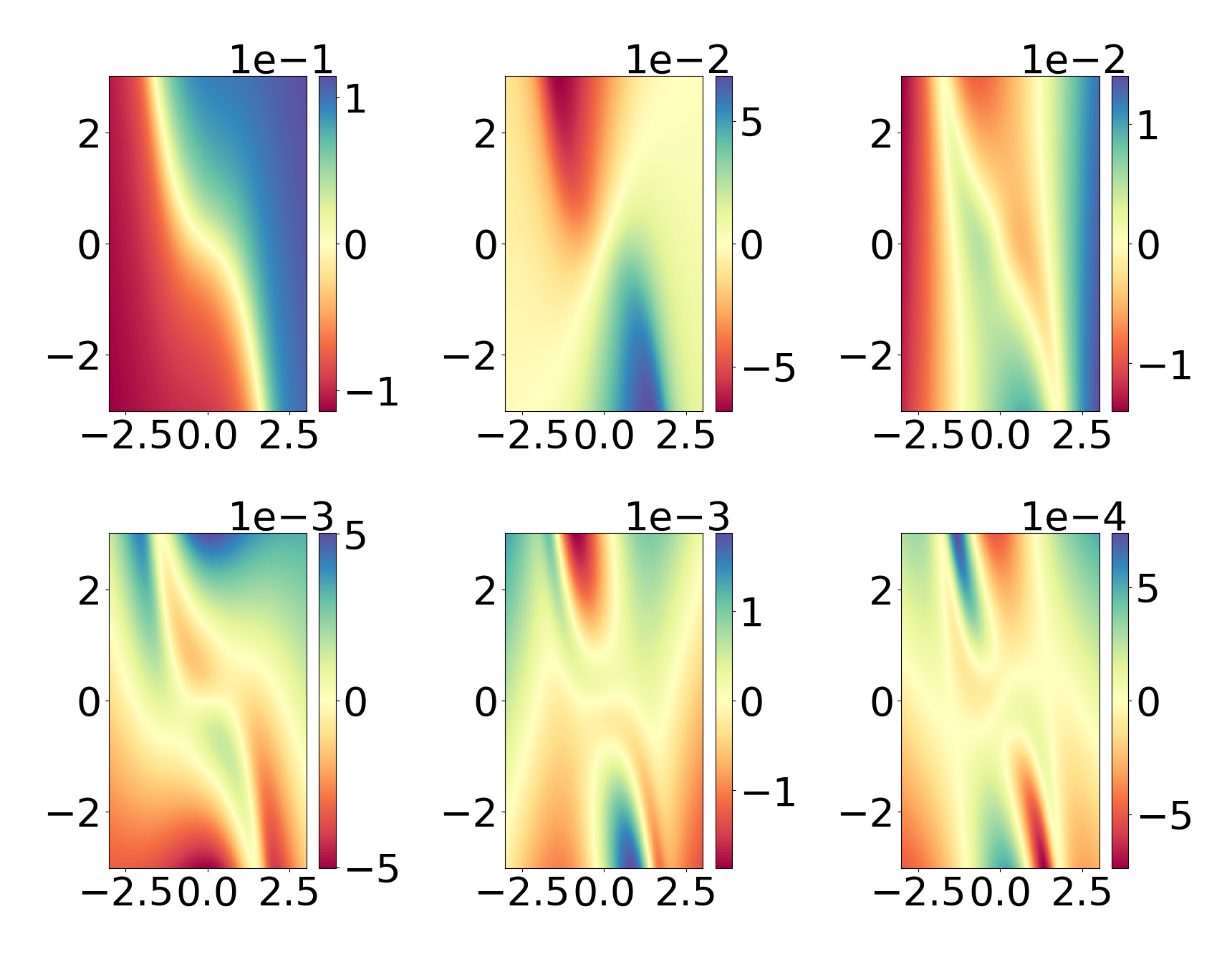}
	\end{minipage}
		\caption{Left: The squared norm of the first eighty eigenfunctions $p_i$ for the nonlinear example. Right: The first six eigenfunctions $p_i$.}
		\label{fig:nonlinear_eigenfunctions}
	\end{figure}
	\begin{figure}[h]
		\includegraphics[width = .45 \textwidth]{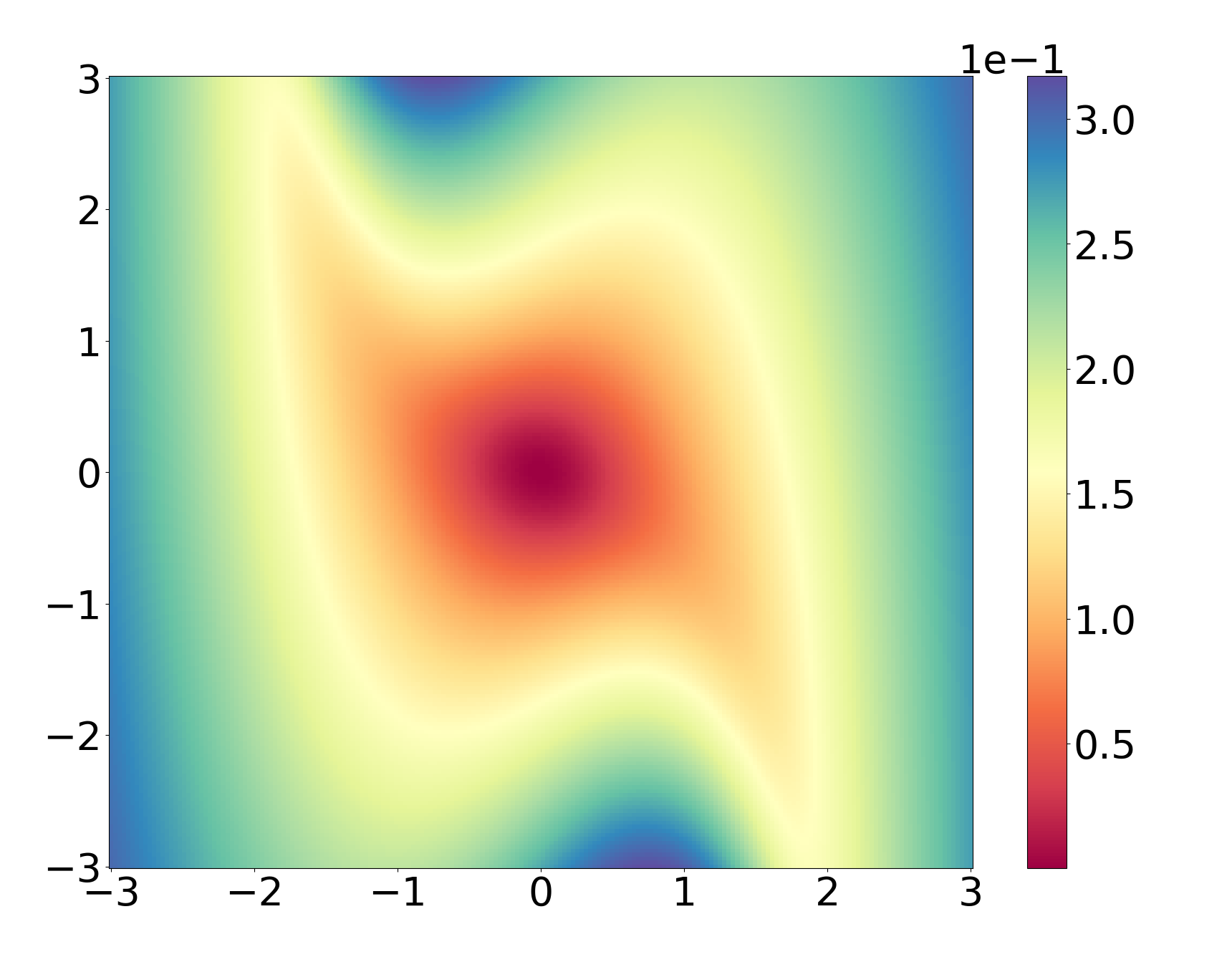}		
		\includegraphics[width = .45 \textwidth]{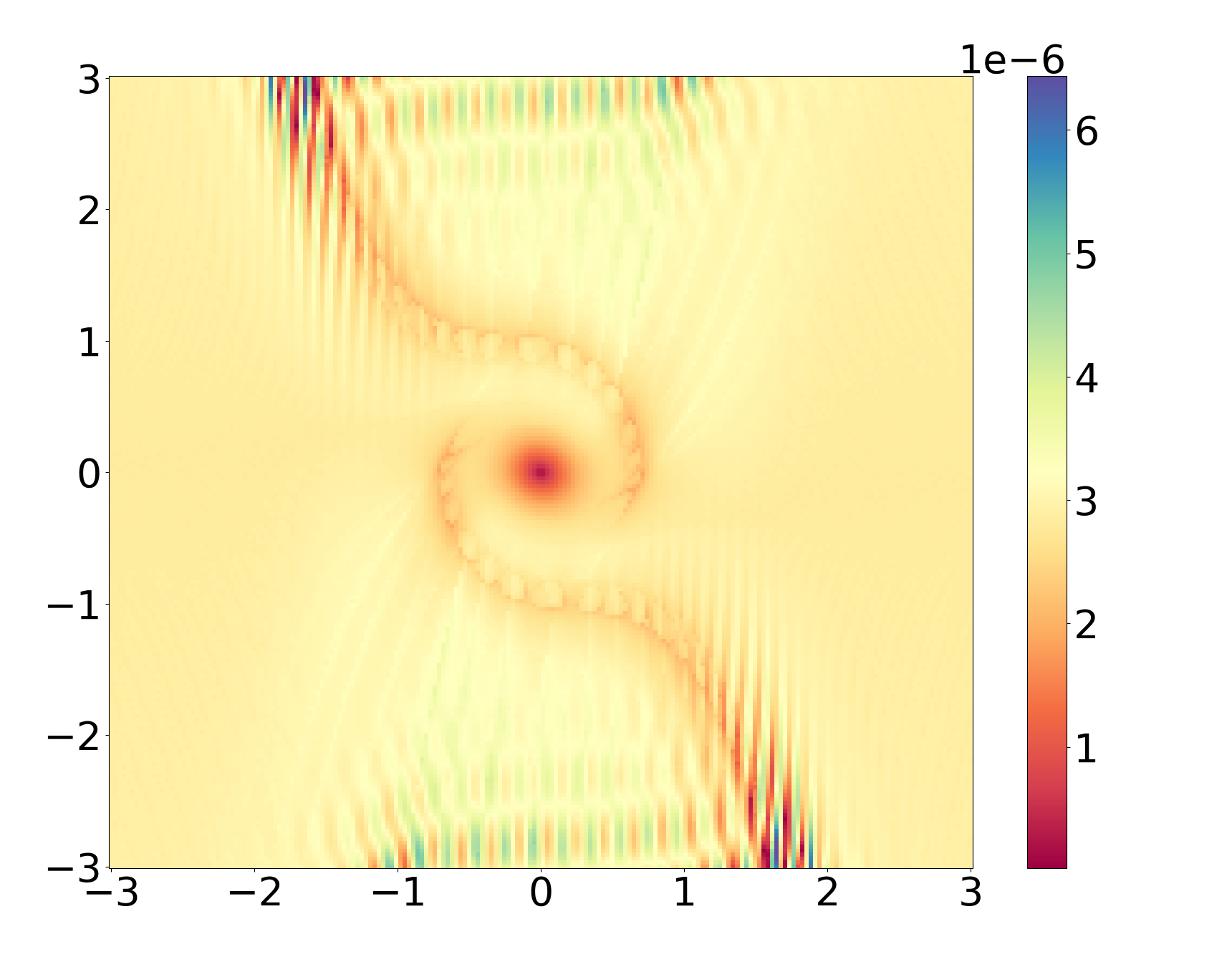}
		\caption{Left: The computed sum of squares solution for the nonlinear example. Right: The error between computed solution and the reference solution obtained by integrating over the trajectories.}
		\label{fig:nonlinear_lyapunov_function}
	\end{figure}
	We need to verify that the assumptions of  \cref{proposition:exponential_decay_stable_in_origin} are satisfied. To do this, we first have to ensure that the linearized problem is locally stable around the origin. We have the decomposition
	\[
	f(x_1,x_2) = A_m\begin{pmatrix} x_1 \\ x_2 \end{pmatrix} + \tilde{f}(x_1,x_2) =\begin{pmatrix}
			0 & 1\\
			-1 & \mu - \eta
	\end{pmatrix} \begin{pmatrix} x_1 \\ x_2 \end{pmatrix} + \tilde{f}(x_1,x_2) 
	\]
	where the eigenvalues of the matrix $A_{m}$ are given by
	\[
		\lambda_i(A_m) := \frac{1}{2} \left( p \pm \sqrt{ p^2 - 4} \right)	\qquad \text{with} \; p := \mu - \eta < 0.
	\]
	We immediately see that $\text{Re}(\lambda_i(A_m)) < 0$ for our choice of parameters and therefore the matrix is stable.
	As a reference solution, we approximate the Lyapunov function by integrating the cost along solution trajectories of the system. We use orthonormalized splines with $60$ nodes and degree $4$ for discretization and Gauss-Legendre quadrature of degree $4$ for integration on each subinterval. The rapid decay predicted in \cref{thm:spectral_decay_cinfty} can be seen in  \cref{fig:nonlinear_eigenfunctions}, along with the highly nonlinear eigenfunctions. The error between the reference solution and our method has a magnitude of around $10^{-5}$, as shown in \cref{fig:nonlinear_lyapunov_function}. We attribute this error at least partially to the way we compute the reference solution.
	
	\section{Conclusion and outlook}
	\label{sec:conclusion}
	
	In this paper, we presented a method for representing a Lyapunov function as the solution to an operator Lyapunov equation. We showed that the solution to this operator equation has a nuclear decomposition with rapidly decaying singular values which allows for a low-rank approximation. We demonstrated the feasibility of this approximation both theoretically and numerically.
	
	Several aspects seem to be worth to be investigated further, one of them being the extension of our concepts to the case of (high-dimensional) nonlinear control problems which are often solved via a sequence of Lyapunov equations in the policy iteration.
	Moreover, we believe our results to be also applicable in the context of model order reduction where the linear structure of the infinite-dimensional system could be used for balanced truncation like techniques. 

\section*{Acknowledgement}

 We thank M.~Oster (TU Berlin) for helpful comments and discussions on an earlier version of this manuscript.

	\bibliography{references} 

\begin{thebibliography}{10}

\bibitem{Alt12}
{\sc H.~W. Alt}, {\em Linear functional analysis}, Springer-Verlag London,
  2016.
\newblock An application-oriented introduction, Translated from the German
  edition by Robert N\"{u}rnberg.

\bibitem{AntSZ02}
{\sc A.~C. Antoulas, D.~C. Sorensen, and Y.~Zhou}, {\em On the decay rate of
  the {H}ankel singular values and related issues}, Systems \& Control Letters,
  46 (2002), pp.~323--342.

\bibitem{AzmKK21}
{\sc B.~Azmi, D.~Kalise, and K.~Kunisch}, {\em Optimal feedback law recovery by
  gradient-augmented sparse polynomial regression}, Journal of Machine Learning
  Research, 22 (2021), pp.~1--32.

\bibitem{BenLP08}
{\sc P.~Benner, J.-R. Li, and T.~Penzl}, {\em Numerical solution of large-scale
  {L}yapunov equations, {R}iccati equations, and linear-quadratic optimal
  control problems}, Numerical Linear Algebra with Applications, 15 (2008),
  pp.~755--777.

\bibitem{BenS13}
{\sc P.~Benner and J.~Saak}, {\em Numerical solution of large and sparse
  continuous time algebraic matrix {R}iccati and {L}yapunov equations: a state
  of the art survey}, GAMM-Mitteilungen, 36 (2013), pp.~32--52.

\bibitem{Bruetal22}
{\sc S.~L. Brunton, M.~Budi\v{s}i\'{c}, E.~Kaiser, and J.~N. Kutz}, {\em Modern
  {K}oopman theory for dynamical systems}, SIAM Review, 64 (2022),
  pp.~229--340.

\bibitem{BudMM12}
{\sc M.~Budi\v{s}i\'{c}, R.~Mohr, and I.~Mezi\'{c}}, {\em Applied
  {K}oopmanism}, Chaos: An Interdisciplinary Journal of Nonlinear Science, 22
  (2012), p.~047510.

\bibitem{Car32}
{\sc T.~Carleman}, {\em Application de la th{\'e}orie des {\'e}quations
  int{\'e}grales lin{\'e}aires aux syst{\`e}mes d'{\'e}quations
  diff{\'e}rentielles non lin{\'e}aires}, Acta Mathematica, 59 (1932),
  pp.~63--87.

\bibitem{CurZ95}
{\sc R.~Curtain and H.~Zwart}, {\em An Introduction to Infinite-Dimensional
  Linear Systems Theory}, Springer, New York, 1995.

\bibitem{CurS01}
{\sc R.~F. Curtain and A.~J. Sasane}, {\em Compactness and nuclearity of the
  {H}ankel operator and internal stability of infinite-dimensional state linear
  systems}, International Journal of Control, 74 (2001), pp.~1260--1270.

\bibitem{EngN06}
{\sc K.~Engel and R.~Nagel}, {\em A short course on operator semigroups},
  Springer, 2006.

\bibitem{Eva98}
{\sc L.~C. Evans}, {\em Partial Differential Equations}, American Mathematical
  Society, 1998.

\bibitem{FroJK18}
{\sc G.~Froyland, O.~Junge, and P.~Koltai}, {\em Estimating long term behavior
  of flows without trajectory integration: the infinitesimal generator
  approach}, SIAM Journal on Numerical Analysis, 51 (2013), pp.~223--247.

\bibitem{GotO77}
{\sc D.~Gottlieb and S.~A. Orszag}, {\em Numerical Analysis of Spectral
  Methods}, Society for Industrial and Applied Mathematics, 1977.

\bibitem{Gra04}
{\sc L.~Grasedyck}, {\em Existence and computation of low {K}ronecker-rank
  approximations for large linear systems of tensor product structure},
  Computing, 72 (2004), pp.~247--265.

\bibitem{Gron19}
{\sc T.~H. Gronwall}, {\em Note on the derivatives with respect to a parameter
  of the solutions of a system of differential equations}, Annals of
  Mathematics, 20 (1919), pp.~292--296.

\bibitem{GruK14}
{\sc L.~Grubisic and D.~Kressner}, {\em On the eigenvalue decay of solutions to
  operator {L}yapunov equations}, Systems \& Control Letters, 73 (2014),
  pp.~42--47.

\bibitem{Kluetal20}
{\sc S.~Klus, F.~N\"uske, S.~Peitz, J.-H. Niemann, C.~Clementi, and
  C.~Sch\"utte}, {\em Data-driven approximation of the {K}oopman generator:
  Model reduction, system identification, and control}, Physica D: Nonlinear
  Phenomena, 406 (2020), p.~132416.

\bibitem{Koo31}
{\sc B.~O. Koopman}, {\em Hamiltonian systems and transformation in {H}ilbert
  space}, Proceedings of the National Academy of Sciences, 17 (1931),
  pp.~315--318.

\bibitem{LasM94}
{\sc A.~Lasota and M.~C. Mackey}, {\em Chaos, Fractals, and Noise: Stochastic
  Aspects of Dynamics}, Springer New York, New York, NY, 1994.

\bibitem{MauM16}
{\sc A.~Mauroy and I.~Mezi{\'c}}, {\em Global stability analysis using the
  eigenfunctions of the {K}oopman operator}, IEEE Transactions on Automatic
  Control, 61 (2016), pp.~3356--3369.

\bibitem{MauSM20}
{\sc A.~Mauroy, Y.~Susuki, and I.~Mezi{\'c}}, {\em The {K}oopman operator in
  systems and control. Concepts, methodologies and applications}, Cham:
  Springer, 2020.

\bibitem{Mez05}
{\sc I.~Mezi{\'c}}, {\em Spectral properties of dynamical systems, model
  reduction and decompositions}, Nonlinear Dynamics, 41 (2005), pp.~309--325.

\bibitem{OpmRW13}
{\sc M.~Opmeer, T.~Reis, and W.~Wollner}, {\em Finite-rank {ADI} iteration for
  operator {L}yapunov equations}, SIAM Journal on Control and Optimization, 51
  (2013).

\bibitem{Opm10}
{\sc M.~R. Opmeer}, {\em Decay of {H}ankel singular values of analytic control
  systems}, Systems \& Control Letters, 59 (2010), pp.~635--638.

\bibitem{Opm15}
{\sc M.~R. Opmeer}, {\em Decay of singular values of the {G}ramians of
  infinite-dimensional systems}, in 2015 European Control Conference (ECC),
  2015, pp.~1183--1188.

\bibitem{Opm20}
{\sc M.~R. Opmeer}, {\em Decay of singular values for infinite-dimensional
  systems with {G}evrey regularity}, Systems \& Control Letters, 137 (2020),
  p.~104644.

\bibitem{Pen00}
{\sc T.~Penzl}, {\em Eigenvalue decay bounds for solutions of {L}yapunov
  equations: the symmetric case}, Systems \& Control Letters, 40 (2000),
  pp.~139--144.

\bibitem{Rud76}
{\sc W.~Rudin}, {\em Principles of Mathematical Analysis}, International series
  in pure and applied mathematics, McGraw-Hill, 1976.

\bibitem{Rud87}
{\sc W.~Rudin}, {\em Real and complex analysis}, McGraw-Hill, 1987.

\bibitem{Sch93}
{\sc J.~M.~A. Scherpen}, {\em Balancing for nonlinear systems}, Systems \&
  Control Letters,  (1993), pp.~143--153.

\bibitem{Sim07}
{\sc V.~Simoncini}, {\em A new iterative method for solving large-scale
  {L}yapunov matrix equations}, SIAM Journal on Scientific Computing, 29
  (2007), pp.~1268--1288.

\bibitem{Sim16}
\leavevmode\vrule height 2pt depth -1.6pt width 23pt, {\em Computational
  methods for linear matrix equations}, SIAM Review, 58 (2016), pp.~377--441.

\bibitem{SinM93}
{\sc R.~K. Singh and J.~S. Manhas}, {\em Composition operators on function
  spaces}, vol.~179 of North-Holland Mathematics Studies, North-Holland
  Publishing Co., Amsterdam, 1993.

\bibitem{Sze39}
{\sc G.~Szegö}, {\em Orthogonal Polynomials}, vol.~23 of Colloquium
  Publications, American Mathematical Society, 1939.

\bibitem{TucW09}
{\sc M.~Tucsnak and G.~Weiss}, {\em Observation and Control for Operator
  Semigroups}, Birkh{\"a}user Basel, 2009.

\bibitem{Wal98}
{\sc W.~Walter}, {\em Ordinary Differential Equations}, Springer New York,
  1998.

\end{thebibliography}
	\bibliographystyle{siam}
	
\end{document}